\documentclass[10pt]{amsart}
\usepackage[utf8]{inputenc}  
\usepackage{amsmath, amssymb, mathtools,amsfonts}
\usepackage[margin=3cm]{geometry}
\usepackage[pagewise]{lineno}
\usepackage{xcolor}
\usepackage{graphicx} 
\usepackage{tikz}
\usetikzlibrary{decorations.markings,backgrounds}
\usepackage{pgfplots}
\pgfplotsset{compat=1.18}
\usepackage{float}
\usepackage{caption,subcaption}
\usepackage{esint}
\usepackage{bbold}
\usepackage{dsfont}
\usepackage{bbm}
\usepackage{multicol}
\usepackage{hyperref}
\usepackage{cleveref}

\allowdisplaybreaks

\newtheorem{theo}{Theorem}
\newtheorem{prop}{Proposition}
\newtheorem{lemma}{Lemma}

\newtheorem{cor}{Corollary}

\newcommand{\R}{\mathbb{R}}	
\newcommand{\C}{\mathbb{C}}	
\newcommand{\N}{\mathbb{N}}	
\newcommand{\eps}{\varepsilon}	
\newcommand{\pa}{\partial}		
\newcommand{\Div}{\textrm{div}\,}	

\newcommand{\na}{\nabla}		

\newcommand{\chf}[1]{{\raisebox{3pt}{\Large $\chi$}}_{#1}}
\newcommand{\HH}{\mathbb{H}}
\newcommand{\dd}{\,\mathrm{d}}
\renewcommand{\d}{\mathrm{d}}
\newcommand{\Lea}{\mathcal{L}_{\varepsilon}^\alpha}
\newcommand{\Le}{\mathcal{L}_{\varepsilon}}
\newcommand{\uMrek}{u^{(M,\rho,\eps,\kappa)}}
\newcommand{\urek}{u^{(\rho,\eps,\kappa)}}
\newcommand{\uek}{u^{(\eps,\kappa)}}
\newcommand{\vek}{v^{(\eps,\kappa)}}
\newcommand{\uk}{u^{(\kappa)}}
\newcommand{\vk}{v^{(\kappa)}}

\title[Fractional cross-diffusion in a bounded domain]{Fractional cross-diffusion  in a bounded domain: existence, weak-strong uniqueness, and long-time asymptotics}

\author[N.~De Nitti]{Nicola De Nitti}
\address[N.~De Nitti]{Università di Pisa, Dipartimento di Matematica, Largo Bruno Pontecorvo 5, 56127 Pisa, Italy.}
\email[]{nicola.denitti@unipi.it}

\author[N.~Zamponi]{Nicola Zamponi}
\address[N.~Zamponi]{Universit\"{a}t Augsburg, Institut f\"ur Mathematik, Universit\"{a}tsstra\ss e 12a, 86159 Augsburg, Germany.}
\email[]{nicola.zamponi@uni-a.de}

\begin{document}

\begin{abstract}
		We study a fractional cross-diffusion system that describes the evolution of multi-species populations in the regime of large-distance interactions in a bounded domain. We prove existence and weak-strong uniqueness results for the initial-boundary value problem and analyze the convergence of the solutions to equilibrium via relative entropy methods.
\end{abstract}

\maketitle

\section{Introduction}
\label{sec:intro}

The study of dispersal strategies and the comparison between local and nonlocal diffusive behaviors of interacting populations have recently attracted great attention (see, e.\,g., \cite{MR3082317, MR3590678, viswanathan1996levy, humphries2010environmental} and references therein). In particular, we shall focus on the  following \emph{fractional cross-diffusion system} modeling the dynamics of interacting \emph{multi-species populations}:
\begin{align}
&\label{1} \pa_t u_i = \Div \sum_{j=1}^{n} a_{ij}u_i\na (-\Delta)^{-\frac{1-\beta}{2}}u_j, && (t,x) \in  (0,+\infty)\times  \Omega,\\
&\label{1.ic}  u_i(0,x) = u_{0,i}(x),  && x \in \Omega, \\ 
&\label{1.bc}
 \nu\cdot\Big(\sum_{j} a_{ij} u_i(t,x)\na (-\Delta)^{-\frac{1-\beta}{2}}u_j(t,x)\Big)  = 0, && (t,x) \in (0,+\infty) \times \partial \Omega, 
\end{align}
for $i\in\{1,\ldots,n\}$. 

The equations are to be solved with respect to $u : (0,+\infty) \times \Omega\to [0,\infty)^n$. Here $\Omega\subset\R^d$ is a bounded domain with smooth boundary, $\nu$ is the outward normal vector to $\partial \Omega$,  $d\in \{1, 2, 3\}$, $0<\beta<1$, $a_{ij}\geq 0$ for $i,j\in\{1,\ldots,n\}$ are constants. We define the \emph{inverse (spectral) fractional Laplacian} $(-\Delta)^{-s}$, $0<s<1$, via its \emph{spectral decomposition}  as follows:
\begin{align}\label{frac.L}
(-\Delta)^{-s}\phi \coloneqq  \sum_{k=1}^{\infty}\lambda_k^{-s}(\psi_k,\phi)_{L^2(\Omega)}\psi_k,\qquad\text{ for all } \phi\in H^2(\Omega),
\end{align}
where $(\cdot,\cdot)_{L^2(\Omega)}$ represents the $L^2(\Omega)$ scalar product, and $(\psi_k)_{k\in\N}\subset L^2(\Omega)$ is a complete orthonormal system for $L_0^2(\Omega)\equiv \{ u\in L^2(\Omega):~\int_\Omega u \dd x =0 \}$ composed of eigenfunctions of $-\Delta$ (with homogeneous Neumann boundary conditions) with eigenvalues $(\lambda_k)_{k\in\N}\subset (0,+\infty)$:
\begin{align}\label{eigenf}
\begin{cases} 
-\Delta\psi_k = \lambda_k\psi_k & x \in \Omega,\\
\partial_\nu \psi_k = 0 & x\in \pa\Omega,
\end{cases} \qquad k\geq 1.
\end{align}

We remark that considering zero-average subset of $L^2$ is needed to define the negative powers of the Neumann Laplacian and recall that then the eigenvalues $\lambda_k$, $k\geq 1$, are all positive and form an increasing sequence which tends to $\infty$ as $k\to\infty$.

Equivalently, the following \emph{heat semigroup representation} of the Neumann spectral fractional Laplacian (and its inverse defined in \eqref{frac.L}) holds:
\begin{align}\label{eq:sfl_sg}
(-\Delta)^{s}u(x) 
&= \frac{s}{\Gamma(1-s)} \int_0^{+\infty} \left(u(x) -e^{t\Delta} u(x)\right) \, \frac{\d t}{t^{1+s}}&&\text{ for all } u\in H^{2s}(\Omega),\\
\label{eq:sfl_sg.inv}
(-\Delta)^{-s}u(x) &= \frac{1}{\Gamma(s)} \int_0^\infty e^{t\Delta}u(x) \frac{\dd t}{t^{1-s}}&&\text{ for all } u\in L^2_0(\Omega),
\end{align}
where 
$e^{t\Delta}u(x)$ is the solution of the boundary value problem for the heat equation 	 
\begin{align}\label{eq:heat}
\begin{cases}
\partial_t w(t,x) = \Delta w(t,x), & (t,x) \in (0,+\infty)\times \Omega, \\
\partial_\nu w(t,x) = 0, & (t,x) \in  (0,+\infty) \times \partial \Omega, \\
w(0,x) = u(x), & x \in \Omega.
\end{cases}
\end{align}
Formula \eqref{eq:sfl_sg} is contained in \cite{MR4080918}. On the other hand, for $s<0$, the integral in \eqref{eq:sfl_sg} is not convergent since $t^{-1-s}$ is not integrable at $\infty$ while the difference $u(x) -e^{t\Delta} u(x)$ tends to $u(x) - \fint_\Omega u \dd x \neq 0$ (unless $u$ is constant); this motivates formula \eqref{eq:sfl_sg.inv}, which can be verified by projecting it along a complete orthonormal system of $L^2_0(\Omega)$ composed by  eigenfunctions of $-\Delta$.

In what follows, we shall assume that the matrix $(a_{ij})_{i,j\in\{1,\ldots,n\}}$ satisfies the  \emph{detailed balance condition}, i.\,e., 
\begin{align}\label{detbal}
\exists\pi_1,\ldots,\pi_n>0:\quad \pi_i a_{ij} = \pi_j a_{ji},\quad i,j\in\{1,\ldots,n\},
\end{align}
and that its spectrum is entirely contained in the half-plane $\{z\in\C:~ \Re(z)>0\}$. Under these assumptions, the matrix $\pi_i a_{ij}$ is symmetric and positive definite. Furthermore, system \eqref{1}--\eqref{1.bc} admits the following \emph{energy functional} (see \cite{MR3350617,MR4199391} for further details):
\begin{align}\label{en}
H[u] \coloneqq  \int_\Omega \sum_{i=1}^n \pi_i u_i\log u_i \dd x.
\end{align}
Indeed, we have 
\begin{align}\label{en.in}
\frac{\d}{\d t}H[u] + c_0\int_\Omega\sum_{i=1}^n |\na (-\Delta)^{-\frac{1-\beta}{4}}u_i|^2 \dd x\leq 0,\qquad t>0,
\end{align}
where $c_0>0$ is the smallest eigenvalue of the matrix $(\pi_i a_{ij})_{i,j=1,\ldots n}$ (see Theorem \ref{thm:ex} for further details).

We will prove global-in-time existence of weak solutions to \eqref{1}--\eqref{1.bc}. 

We  stress the fact that the boundary condition \eqref{1.bc} is to be interpreted in light of the weak formulation of \eqref{1}--\eqref{1.bc} (see Theorem \ref{thm:ex} below). 
The choice of the no-flux boundary conditions \eqref{1.bc} ensures mass conservation (as the constant test function $\phi\equiv 1$ can be chosen in the weak formulation of \eqref{1}--\eqref{1.bc}). 

Moreover, given the definition of the fractional Laplacian \eqref{frac.L} and the eigenfunctions \eqref{eigenf}, the boundary condition \eqref{1.bc} is automatically satisfied by any smooth enough function $u$: indeed, we formally compute 
\begin{align*}
 \nu\cdot\Big(\sum_{j} a_{ij} u_i(t,x)\na (-\Delta)^{-\frac{1-\beta}{2}}u_j(t,x)\Big)  &= \sum_{j} a_{ij} u_i(t,x)\nu \cdot \na (-\Delta)^{-\frac{1-\beta}{2}}u_j(t,x)
 \\&= \sum_{j} a_{ij} u_i(t,x)\nu \cdot \Big( \na \sum_{k=1}^{\infty}\lambda_k^{-\frac{1-\beta}{2}}(\psi_k,u_j(t,\cdot))_{L^2(\Omega)}\psi_k(x) \Big)
  \\&= \sum_{j} a_{ij} u_i(t,x)  \sum_{k=1}^{\infty}\lambda_k^{-\frac{1-\beta}{2}}(\psi_k,u_j(t,\cdot))_{L^2(\Omega)}\partial_\nu \psi_k(x) \\ & = 0, 
\end{align*} 


In $\mathbb R^d$, the system 
\begin{align}
\label{eq:cdRn}&  \pa_t u_i + \sigma_i (-\Delta)^\alpha u_i -  \Div \sum_{j=1}^{n} a_{ij}u_i\na (-\Delta)^{-\frac{1-\beta}{2}}u_j = 0, && (t,x) \in  (0,+\infty)\times \R^d,\\
\label{eq:cdRn-ic}&   u_i(0,x) = u_{0,i}(x),  && x \in \R^d
\end{align}
has been derived in \cite{MR4346700} as the \emph{many-particle limit} of the following \emph{interacting particle system driven by Lévy noise}:
$$\d X_{i}^{k, N}(t)=-\sum_{j=1}^{n} \frac{1}{N} \sum_{\ell=1}^{N} a_{i j} \nabla(-\Delta)^{(\beta-1) / 2} V_{N}\left(X_{i}^{k, N}(t)-X_{j}^{\ell, N}(t)\right) \dd t+\sqrt{2 \sigma_{i}} \dd L_{i}^{k}(t),$$
where $i\in \{1, \ldots, n\}$ and $k\in \{1, \ldots, N\}$, $X_{i}^{k, N}(t)$ is the position of the $k$-th particle of species $i$ at time $t$, $V_{N}$ is a suitably chosen potential function, and $L_{i}^{k}$ is a Lévy process of index $\alpha \in(0,1)$.

The global-in-time existence of strong solutions of \eqref{eq:cdRn}--\eqref{eq:cdRn-ic} is proved in  \cite{MR4346700} for sufficiently small initial data in $H^{s}\left(\mathbb{R}^{d}\right)$ (with $s>d / 2$) in the regime $2 \alpha>\beta+1$, i.\,e., when self-diffusion dominates cross-diffusion; on the other hand, in \cite{MR4401974}, global-in-time existence of weak solutions is proved without any smallness assumption on the initial data and for all values of $\alpha, \beta \in(0,1)$ (also allowing for the case without self-diffusion, namely $\sigma_{i}=0$, on which we focus in the present paper). 

The convergence to equilibrium for reaction-diffusion systems with linear diffusion has been studied extensively (see, e.\,g., \cite{MR2217853} and references therein); in the case of nonlinear diffusion or cross-diffusion, some results are available for diffusion of porous medium type \cite{MR4142243}, Maxwell--Stefan diffusion \cite{MR4064194}; reaction-diffusion systems \cite{MR3785634}, and reaction-cross-diffusion systems with strongly growing complex balanced reactions coming from mass-action kinetics \cite{MR3862716}. 

However, for the case of fractional cross-diffusion, the study of the initial-boundary-value problem seems to be widely open. In the present paper, we address this gap in the literature, showing the existence of a weak solution of \eqref{1}--\eqref{1.bc}, a weak-strong uniqueness principle, and the convergence to equilibrium as time goes to infinity. 

The main difficulty of the problem under consideration lies in the simultaneous presence and intertwinement of cross-diffusion and fractional diffusion in Eq.~\eqref{1}. This fact leads to a very rigid structure for system \eqref{1}--\eqref{1.bc}, which critically limits the analytical techniques that can be effectively applied to study the problem.
In \cite{MR4401974} a similar problem was considered, with similar analytical issues, but in comparison to \cite{MR4401974}, where only an existence analysis is performed, we study existence of weak solutions, long-time behavior and weak-strong uniqueness of solutions for \eqref{1}--\eqref{1.bc}. An additional novelty of our work is the fact that we consider a problem on a bounded domain and work with the spectral fractional Laplacian, a framework which (as already mentioned) has not yet been considered in the literature.

For the case of one component (namely, when the system reduces to a porous medium equation with nonlocal pressure), more results are available (see, e.\,g., \cite{zbMATH06101958,zbMATH06552524,zbMATH07773346,zbMATH06759326}). We refer, in particular, to \cite{zbMATH06482496,MR3916644}, where the case of the spectral fractional Laplacian is treated.

To prove the existence of solution, we  combine the approach of \cite{MR3916644} with (an adaptation of) the regularization strategy of \cite{MR4401974}. 
Precisely, we carefully design an elaborated approximation scheme which is adapted to the rigid fractional cross-diffusion structure of the system, while we obtain compactness for the approximating sequence by extending the usual results on Sobolev embeddings and integral inequalities to the setting of fractional Sobolev spaces defined via the spectral fractional Laplacian on bounded domains.

In the proof of weak-strong uniqueness and convergence to equilibrium, a pivotal role is played by the relative entropy method. This technique was introduced by DiPerna (see \cite{MR513809,MR523630}) and Dafermos (see \cite{MR546634,doi:10.1080/01495737908962394}) in order to estimate the $L^2$-distance between two solutions of hyperbolic systems of conservation laws, one of both being smooth or at least not too weak. We refer to \cite{MR2909912,MR3729430} for further applications of this method to fluid dynamics (in particular, to the Navier--Stokes system).
In this work we consider the relative entropy between a standard weak solution $u$ to the system and a hypothetical strong solution $v$, with both solutions coinciding at initial time, and estimate its time derivative (``entropy dissipation'') via a fractional Gagliardo--Nirenberg--Sobolev inequality (proved in the paper). While the basic idea of the method is not new, its successful application in the context of fractional cross-diffusion equations is novel and nontrivial, as it requires the extension of the standard embeddings and functional inequalities to the setting of (spectral) fractional Sobolev spaces, and it involves working with a rigid structure and with pseudo-differential operators for which the simple properties of classical (and weak) derivatives (e.\,g.,~chain rule, differentiation of products etc) do not apply. The weak solutions that we consider (which we obtain from the existence theorem) are also weaker than the standard ones, being $L^2(0,T; H^\beta(\Omega))$ with $0<\beta<1$ instead of $L^2(0,T; H^1(\Omega))$.

\subsection{Outline of the paper}
\label{ssec:outline}

The paper is organized as follows. In Section \ref{sec:main}, we state our three main results: existence of global weak solutions for \eqref{1}--\eqref{1.bc}, weak-strong uniqueness, and long-time behavior. In Section \ref{sec:prelim}, we introduce the notation and preliminary results that are needed. The main results are then proved in Section \ref{sec:thm-ex}, \ref{sec:thm-wsuniq}, and \ref{sec:thm-lt} respectively.

\section{Main results}
\label{sec:main}

Our first main result concerns the existence of weak solutions to \eqref{1}--\eqref{1.bc}. 

\begin{theo}[Global existence of weak solutions]\label{thm:ex}
	Let us assume that the initial data $(u_{0,1}, \dots, u_{0,n}) \eqqcolon u_0 : \Omega\to [0,+\infty)^n$ satisfies $\int_\Omega u_i|\log u_i|\dd x < \infty$ for $i\in\{1,\ldots,n\}$. Then
	system \eqref{1}--\eqref{1.bc} has a weak solution $u : \Omega\times (0,+\infty)\to\R^n$ with non-negative components satisfying, for every $T>0$,
	\begin{align*}
	u_i\in L^\infty((0,T); L^1(\Omega)),\quad
	\nabla (-\Delta)^{-\frac{1-\beta}{4}}u_i\in L^2((0,T); L^2(\Omega)), \quad
	\pa_t u_i\in L^r((0,T); W^{1,r}(\Omega)'),
	\end{align*}
	for $i\in\{1,\ldots,n\}$, with a suitable $r>1$, and the weak formulation of \eqref{1}--\eqref{1.bc},
	\begin{align*}
	\int_0^T\langle\pa_t u_i,\phi\rangle \dd t + \sum_{j=1}^n a_{ij}\int_0^T\int_\Omega \na\phi\cdot u_i\na (-\Delta)^{-\frac{1-\beta}{2}}u_j \dd x \dd t = 0,\quad i\in\{1,\ldots n\},\\
	u_i(t,\cdot)\to u_{0,i}\quad\mbox{strongly in $C^0([0,T], [W^{1,r}(\Omega)]')$ as $t\to 0$,}\quad i\in\{1,\ldots n\},
	\end{align*}
	for every $\phi\in L^{r}((0,T); W^{1,r}(\Omega))$, where $\langle\cdot,\cdot\rangle$ is the duality product between $[W^{1,r}(\Omega)]'$ and $W^{1,r}(\Omega)$. 
	
Moreover, $u$ satisfies the following \emph{entropy inequality}: there exists a constant $C>0$ such that 
\begin{align}\label{en.in.i}
H[u(t,\cdot)] + c_0\int_0^t\int_\Omega\sum_{i=1}^n |\na (-\Delta)^{-\frac{1-\beta}{4}}u_i|^2 \dd x\leq C,\qquad t>0,
\end{align}	
where $c_0>0$ is the smallest eigenvalue of the matrix $(\pi_i a_{ij})_{i,j=1,\ldots n}$. 
\end{theo}

Secondly, we prove a weak-strong uniqueness principle: any weak solution coincides with a classical solution (or, more accurately, with a weak solution whose logarithmic derivative satisfies further integrability properties) \textcolor{red}{if the latter exists}. We remark that the existence of such stronger notion of solutions (for short times) remains an open problem for the system \eqref{1}--\eqref{1.bc}.

\begin{theo}[Weak-strong uniqueness]\label{thm:wsuniq}
	Let us assume that \eqref{1}--\eqref{1.bc} has a weak solution $v$ satisfying 
	\begin{align}\label{hp.wsuni}
	\exists q_1, q_2>1,\quad
    \frac{d+1+\beta}{q_1} + \frac{d}{q_2} < 1:\quad
	\nabla\log v_i\in L^{q_1}((0,T); L^{q_2}(\Omega)),
    \quad i\in\{1,\ldots,n\},
	\end{align}
	for every $T>0$.
	Then every weak solution $u$ to \eqref{1}--\eqref{1.bc} (in the sense of Theorem \ref{thm:ex}) coincides with $v$ a.\,e.~in $\Omega$, for $t>0$.
\end{theo}

Finally, we study the long-time asymptotic behavior of the solutions to \eqref{1}--\eqref{1.bc}.

\begin{theo}[Long-time behaviour]\label{thm:lt}
Let us assume that $\Omega$ is connected. Then the weak solutions to \eqref{1}--\eqref{1.bc} (in the sense of  Theorem~\ref{thm:ex}) converge  strongly in $L^1(\Omega)$ towards $u^\infty \equiv \fint_\Omega u_0 \dd x$ with an exponential decay rate as $t\to\infty$.
\end{theo}

\section{Preliminary results}
\label{sec:prelim}

We start by presenting some technical results related to the spectral fractional Laplacian. For $0<r<1$ define the space
\begin{align*}
\HH^r(\Omega) \coloneqq  \{ u\in L^2_0(\Omega)~:~ \nabla (-\Delta)^{(r-1)/2}u\in L^2(\Omega) \},
\end{align*}
which can also be characterized (arguing similarly as \cite[Proposition  2.2, Claim (4)]{MR3397309}) as 
\begin{align*}
\HH^r(\Omega) = \{ u\in L^2_0(\Omega)~:~ \|u\|_{\HH^r(\Omega)}<\infty\},\quad
\|u\|_{\HH^r(\Omega)}^2 = \sum_{j=1}^\infty \lambda_j^r|(\psi_j,u)_{L^2(\Omega)}|^2.
\end{align*}
We start by proving an analogue version of the Sobolev's compact embedding theorem.
\begin{lemma}[Sobolev's embedding]\label{lem.spfrSob}
The embedding $\HH^r(\Omega)\hookrightarrow L^2(\Omega)$ is compact for every $r\in (0,1)$.
\end{lemma}
\begin{proof} Let $(u_n)_{n\in\N}\subset\HH^r(\Omega)$ be a bounded sequence in $\HH^r(\Omega)$. This means
\begin{align*}
\|u_n\|_{\HH^r(\Omega)}^2 = \sum_{j=1}^\infty \lambda_j^r|(\psi_j,u_n)_{L^2(\Omega)}|^2\leq C,\qquad n\in\N.
\end{align*}
Since $\lambda_j\leq \lambda_{j+1}$ for $j\geq 1$, it follows
\begin{align}\label{est.tail}
\sup_{n\in\N}\sum_{j=k}^\infty |(\psi_j,u_n)_{L^2(\Omega)}|^2\leq C\lambda_k^{-r}\to 0\quad\mbox{as }k\to\infty.
\end{align}
Being $\lambda_j\geq\lambda_1>0$ for $j\geq 1$, we have that $(u_n)_{n\in\N}$ is bounded in $L^2(\Omega)$; therefore, for every $j\in\N$, the sequence of real numbers $( (\psi_j,u_n)_{L^2(\Omega)} )_{n\in\N}$ is bounded and thus relatively compact in $\R$. A Cantor diagonal argument allows us to extract a subsequence (not relabeled) of $(u_n)_{n\in\N}$ such that
\begin{align}\label{cnv.Sob}
(\psi_j,u_n)_{L^2(\Omega)}\to \hat{u}_j\quad\mbox{as }n\to\infty,~~\mbox{for }j\geq 1.
\end{align}
Fatou's lemma yields
\begin{align*}
\sum_{j=1}^\infty \lambda_j^r|\hat{u}_j|^2 \leq\liminf_{n\to\infty}
\sum_{j=1}^\infty \lambda_j^r|(\psi_j,u_n)_{L^2(\Omega)}|^2 < \infty,
\end{align*}
which implies $u\equiv \sum_{j=1}^\infty \hat{u}_j\psi_j\in \HH^r(\Omega)$. Furthermore, 
\begin{align}\label{est.tail.2}
\sum_{j=k}^\infty |\hat{u}_j|^2\leq C\lambda_k^{-r}\to 0\quad\mbox{as }k\to\infty.
\end{align}
Let us now consider, for an arbitrary $k\geq 2$,
\begin{align*}
\|u_n - u\|_{L^2(\Omega)}^2 \leq \sum_{j=1}^{k-1}( (\psi_j,u_n)_{L^2(\Omega)} - \hat{u}_j )^2 + 
2\sum_{j=k}^\infty( |(\psi_j,u_n)_{L^2(\Omega)}|^2 + |\hat{u}_j|^2 ).
\end{align*}
From \eqref{est.tail}--\eqref{est.tail.2}, we then deduce
\begin{align*}
\limsup_{n\to\infty}\|u_n - u\|_{L^2(\Omega)}^2 \leq C\lambda_k^{-r}\to 0\quad\mbox{as }k\to\infty,
\end{align*}
yielding that $u_n\to u$ strongly in $L^2(\Omega)$. This concludes the proof.
\end{proof}

\begin{lemma}[Sobolev's embedding]\label{Sob.emb.spfr}
	Let $0<r<1$, $p^*=\frac{2d}{d-2r}$.
	The embedding $\HH^r(\Omega)\hookrightarrow L^p(\Omega)$ is compact for every $p\in [1,p^*)$.
\end{lemma}

\begin{proof} 
	
We preliminarily observe that, thanks to Lemma \ref{lem.spfrSob}, it suffices to prove that the embedding $\HH^r(\Omega)\hookrightarrow L^p(\Omega)$ is continuous for every $p\in [1,p^*)$. We employ \eqref{eq:sfl_sg.inv} to show that 
\begin{align}
	\label{aaa.1}
	\|(-\Delta)^{-r/2}v\|_{L^p(\Omega)} \leq C_p \|v\|_{L^2(\Omega)}\qquad\text{ for all } v\in L^2_0(\Omega).
\end{align}
This will yield the claim in the statement of the Lemma. Indeed, given $u\in \HH^r(\Omega)$, define $v = (-\Delta)^{r/2}u$. Then $v\in L^2_0(\Omega)$ due to the definitions of $\HH^r(\Omega)$ and of the fractional Laplacian. Since 
$$ 
(-\Delta)^{-r/2}(-\Delta)^{r/2}w = w\qquad \text{ for all } w\in \HH^r(\Omega)\cap L^2_0(\Omega),
$$
we infer
\begin{align*}
\left\|u - \fint_\Omega u\right\|_{L^p(\Omega)} = \| (-\Delta)^{-r/2}v \|_{L^p(\Omega)}\leq C_p \|v\|_{L^2(\Omega)} = C_p \|(-\Delta)^{r/2}u\|_{L^2(\Omega)},
\end{align*}
which proves the claim. 

Let us now show \eqref{aaa.1}. Let $s = r/2\in (0,1/2)$. From \eqref{eq:sfl_sg.inv}, we deduce 
\begin{align*}
\|(-\Delta)^{-s}v\|_{L^p(\Omega)}\leq \frac{1}{\Gamma(s)} \int_0^\infty t^{s-1}\|e^{t\Delta}v\|_{L^p(\Omega)} \dd t.
\end{align*}
By standard results on the heat equation (obtainable, e.\,g., via energy methods or spectral arguments), we have 
\begin{align}
	\label{heat.est}
\|e^{t\Delta}v\|_{L^\infty((0,+\infty);L^2(\Omega))}^2 + \|\nabla e^{t\Delta}v\|_{L^2((0,+\infty); L^2(\Omega))}^2 &\leq \|v\|_{L^2(\Omega)}^2 ,\\
\label{heat.decay}
\|e^{t\Delta}v\|_{L^2(\Omega)}&\leq e^{-\lambda_1 t}\|v\|_{L^2(\Omega)}, \qquad t>0,
\end{align}
where $\lambda_1>0$ is the smallest eigenvalue of $-\Delta$ on $L^2_0(\Omega)$ (or, equivalently, the proportionality constant in Poincar\'e's inequality). 

Then, Gagliardo--Nirenberg and Poincar\'e's inequalities lead to
\begin{align}
	\label{GN.heat}
\|e^{t\Delta}v\|_{L^p(\Omega)}\leq \|e^{t\Delta}v\|_{L^2(\Omega)}^{1-\theta}\|\nabla e^{t\Delta}v\|_{L^2(\Omega)}^{\theta},\quad \theta = \frac{d(p-2)}{2p},\quad p < \frac{2d}{d-2r} = \frac{2d}{d-4s}.
\end{align}
Therefore, integrating \eqref{GN.heat} in time and employing H\"older's inequality yield
\begin{align*}
\int_0^\infty t^{s-1}\|e^{t\Delta}v\|_{L^p(\Omega)} \dd t &\leq
\int_0^\infty t^{s-1}\|e^{t\Delta}v\|_{L^2(\Omega)}^{1-\theta}\|\nabla e^{t\Delta}v\|_{L^2(\Omega)}^{\theta} \dd t\\
&\leq \left(
\int_0^\infty t^{-\frac{2(1-s)}{2-\theta}} 
\|e^{t\Delta}v\|_{L^2(\Omega)}^{\frac{2(1-\theta)}{2-\theta}}\dd t
\right)^{1-\theta/2}\left(
\int_0^\infty \|\nabla e^{t\Delta}v\|_{L^2(\Omega)}^2 \dd t\right)^{\theta/2}
\end{align*}
so from \eqref{heat.est}, \eqref{heat.decay} we get
\begin{align}\label{aaa.2}
\int_0^\infty t^{s-1}\|e^{t\Delta}v\|_{L^p(\Omega)} \dd t &\leq
\|v\|_{L^2(\Omega)} \left(
\int_0^\infty t^{-\frac{2(1-s)}{2-\theta}} 
e^{-\frac{2\lambda_1(1-\theta)}{2-\theta}t}\dd t
\right)^{1-\theta/2}.
\end{align}
The right-hand side of \eqref{aaa.2} is finite if and only if $\frac{2(1-s)}{2-\theta} < 1$, which (after some elementary computations) can be rewritten as $p < \frac{2d}{d-4s} = p^*$. This concludes the proof.
\end{proof}

\begin{cor}\label{Hminusr.coro} Let $0<r<1$, $q>\frac{2d}{d+2r}$. There exists $C>0$ such that
$$
\|(-\Delta)^{-r/2}f\|_{L^2(\Omega)}\leq C\|f\|_{L^q(\Omega)}
$$
for every $f\in L^q_0(\Omega)$.
\end{cor}
\begin{proof}
From Lemma \ref{Sob.emb.spfr} it follows via duality that $L^q(\Omega)\hookrightarrow (\HH^r(\Omega))^*$ continuously. However the definition of $\|\cdot\|_{\HH^r(\Omega)}$ implies
\begin{align*}
\|f\|_{(\HH^r(\Omega))^*} 
&= \sup\left\{
\int_\Omega f u\, dx~:~ u\in\HH^r(\Omega),~\|u\|_{\HH^r(\Omega)}\leq 1
\right\}\\
&\geq \sup\left\{
\int_\Omega f\, (-\Delta)^{-r/2}v\, dx~:~ v\in L^2_0(\Omega),~\|v\|_{L^2(\Omega)}\leq 1
\right\}\\
&= \sup\left\{
\int_\Omega v\, (-\Delta)^{-r/2}f\, dx~:~ v\in L^2(\Omega),~\|v\|_{L^2(\Omega)}\leq 1
\right\}\\
&=\|(-\Delta)^{-r/2}f\|_{L^2(\Omega)}.
\end{align*}
This finishes the proof of the Corollary.
\end{proof}

The next result states roughly that ``$\nabla^r u = 0$ implies $u= \text{constant}$'' provided that the domain $\Omega$ is connected.
\begin{lemma}\label{lem.spfrL}
Assume $\Omega$ is connected, and let $u\in \HH^r(\Omega)$ such that 
$\nabla (-\Delta)^{(r-1)/2}u=0$ a.\,e.~in $\Omega$. Then $u=\fint_\Omega u \dd x$ a.\,e.~in $\Omega$.
\end{lemma}
\begin{proof}
Since $\nabla (-\Delta)^{(r-1)/2}u=0$ a.\,e.~in $\Omega$ it follows that
a constant $c\in\R$ exists such that 
$(-\Delta)^{(r-1)/2}u=c$ a.\,e.~in $\Omega$. Therefore, for $j\geq 1$ arbitrary, having $\psi_j$ zero average
\begin{align*}
0 = (\psi_j,c)_{L^2(\Omega)} = (\psi_j, (-\Delta)^{(r-1)/2}u)_{L^2(\Omega)} = \lambda_j^{r-1}(\psi_j, u)_{L^2(\Omega)}
\end{align*}
implying that $u=\fint_\Omega u \dd x$ a.\,e.~in $\Omega$.
This concludes the proof.
\end{proof}

Finally, we recall a Poincar\'e-Wirtinger's inequality for the spectral fractional derivative.
\begin{lemma}[Poincar\'e-Wirtinger's inequality for $\HH^r$]\label{lem.Poi}
Let $\Omega$ connected. Then a positive constant $C_p$ exists such that
\begin{align*}
\|u\|_{L^2(\Omega)}\leq C_p\|\nabla (-\Delta)^{(r-1)/2}u\|_{L^2(\Omega)}\quad\text{ for all } u\in \HH^r(\Omega)\cap L^2_0(\Omega).
\end{align*}
\end{lemma}
\begin{proof} 
Let $\phi_j$, $j\geq 1$ be a complete orthonormal system for $L^2_0(\Omega)$ made by eigenfunctions of $-\Delta$ with homogeneous Neumann boundary conditions.
The spectral decomposition of $(-\Delta)^{(r-1)/2}u$ yields
\begin{align*}
\|\nabla (-\Delta)^{(r-1)/2}u\|_{L^2(\Omega)}^2 = 
\sum_{j,k=1}^\infty \lambda_j^{(r-1)/2}\lambda_k^{(r-1)/2}
(u,\phi_j)_{L^2(\Omega)}(u,\phi_k)_{L^2(\Omega)}
(\nabla\phi_j,\nabla\phi_k)_{L^2(\Omega)}
\end{align*}
and since $(\nabla\phi_j,\nabla\phi_k)_{L^2(\Omega)} = \lambda_j\delta_{jk}$
\begin{align*}
\|\nabla (-\Delta)^{(r-1)/2}u\|_{L^2(\Omega)}^2 = 
\sum_{j=1}^\infty \lambda_j^{r}|(u,\phi_j)_{L^2(\Omega)}|^2
\end{align*}
but being $\lambda_j\geq\lambda_1>0$ for every $j\geq 1$ we infer
\begin{align*}
\|\nabla (-\Delta)^{(r-1)/2}u\|_{L^2(\Omega)}^2 \geq \lambda_1^r 
\sum_{j=1}^\infty |(u,\phi_j)_{L^2(\Omega)}|^2 = 
\lambda_1^r \|u\|_{L^2(\Omega)}^2 .
\end{align*}

\end{proof}

\section{Existence of weak solutions}
\label{sec:thm-ex}

To prove the existence result, we shall rely on an approximation approach on four levels: 
\begin{enumerate}
	\item Introduction of a \emph{small viscosity} $\kappa \Delta u_i$; 
	\item \emph{Removal of the singularity in the fractional Laplace operator}: following \cite{MR3916644}, we shall replace $(-\Delta)^{(\beta-1)/2}$ by $(-\Delta)^{-1}{\mathcal L_\eps^{(\beta + 1)/4}\mathcal L_\eps^{(\beta + 1)/4}}$, where, for $\alpha \in (0,1)$ and $\eps >0$, we define
		\begin{align}\label{eq:approx-L}
		\mathcal{L}_{\varepsilon}^\alpha[u](x)\coloneqq 
		\frac{1}{\Gamma(1-\alpha)}
		\int_{\varepsilon}^{\infty}\left(u(x)-e^{t \Delta} u(x)\right) t^{-1-\alpha} \dd t,
		\end{align}
	where $\Gamma(s)$ is the Euler Gamma function: $\Gamma(s) = \int_0^\infty t^{s-1}e^{-t}dt$.
	\item \emph{Zero-average quadratic correction}: We consider the mapping 
\begin{align}\label{eq:g-ap}
	g_\rho[u](x) = \frac{u(x)^2}{1+\rho |u(x)|} - \fint_\Omega \frac{u(y)^2}{1+\rho |u(y)|}\dd y,\quad u\in L^2(\Omega),
\end{align}
which has the following properties
\begin{align}\label{grho.1}
&\|g_\rho[u]\|_{L^1(\Omega)}\leq 2\|u\|_{L^2(\Omega)}^2,\quad 
\|g_\rho[u]\|_{L^2(\Omega)}\leq C(\rho)\|u\|_{L^2(\Omega)},\\
\label{grho.2}
&\|g_\rho[u] - g_\rho[v]\|_{L^2(\Omega)}\leq C(\rho)\|u-v\|_{L^2(\Omega)},
\end{align}
for every $u,v\in L^2(\Omega)$.
\item \emph{Truncation}: we shall truncate the mobility in the cross-diffusion system by the following operator: 
\begin{align*}
T_M(x) = \min\{ \max\{\xi,0\}, M \},\qquad \xi\in\R,
\end{align*}
for $M>0$. Here, we point out the following elementary property of $T_M$:
\begin{align}\label{TM.1}
|T_M(\xi)-T_M(\xi')|\leq C(M)|\xi-\xi'|,\qquad \xi,\,\xi'\in\R.
\end{align}
\end{enumerate}

The addition of an artificial viscosity is a classical trick to establish the well-posedness of the regularized problem by fixed-point methods. In addition to this, we also regularize the singularity of the fractional Laplacian operator by adopting the strategy of \cite{MR3916644}; finally, in order to obtain a control of the approximate solutions in a higher norm, we introduce a quadratic correction (zero-average, in order to preserve mass).  As a technical point, we notice the need to truncate the solution in order to obtain the well-posedness of the approximated problem in (roughly speaking) $H^1$. The passage to the limit is carried out as $M \to +\infty$, $\rho \to 0^+$, $\eps \to 0^+$, and $\kappa \to 0^+$ thanks to suitable (uniform) a priori estimates and an Aubin--Lions-type argument.

Let us collect some properties of the regularized operator $\mathcal L_\eps^\alpha$. 

\begin{prop}[Properties of the approximate operator $\mathcal L^\alpha_\eps$]\label{Prop.Le} Let $\alpha,\eps>0$. The operator $\mathcal L^\alpha_\eps$, defined in \eqref{eq:approx-L}, satisfies the following properties:
\begin{enumerate}
	\item \emph{Generalized positivity (for multi-species populations)}: Let $(b_{ij})_{i,j=1,\dots,n}$ be a symmetric positive definite matrix. Then 
\begin{align}\label{sys.mono}
\sum_{i,j=1}^{n}b_{ij}\int_{\Omega}f_i \Lea [f_j] \dd x \geq 0\quad
\mbox{for every $f_1,\ldots,f_n\in L^2(\Omega)$.}
\end{align}
\item \emph{Self-adjointness:} for every $f, g\in L^2(\Omega)$ it holds 
\begin{align*}
\int_\Omega f \Le^{\alpha} g \, \mathrm dx = \int_\Omega g \Le^{\alpha} f \, \mathrm dx.
\end{align*}
	\item  \emph{Approximation error:} For any \(\delta \in(0,1-\alpha)\), there is \(C_\delta>0\) such that
	\begin{align}\label{Le.err}
	\left\|(-\Delta)^{-1} \mathcal{L}_{\varepsilon}^\alpha[f]-(-\Delta)^{-1+\alpha} f\right\|_{L^2(\Omega)} \leq C_\delta \varepsilon^{1-\alpha-\delta}\|f\|_{L^2(\Omega)},
	\end{align}
	for all \(0<\varepsilon<1\).
\item \emph{Spectral decomposition:} for any $f\in L^2(\Omega)$ it holds
\begin{align}\label{Le.spec}
\Le^\alpha f = \sum_{j=1}^\infty \lambda_j^\alpha g_\alpha(\eps\lambda_j) f_j \phi_j,\quad f_j\equiv \int_\Omega f \phi_j \, \mathrm d x,\quad 
g_\alpha(s)\equiv \int_s^\infty \frac{1-e^{-t}}{\Gamma(1-\alpha)t^{1+\alpha}} \dd t\quad \text{ for all } s\geq 0.
\end{align}
\end{enumerate}
\end{prop}

\begin{proof}
We first prove \eqref{sys.mono}. It holds
\begin{align*}
\sum_{i,j=1}^{n}b_{ij}\int_{\Omega}f_i \Lea [f_j] \dd x &= 
\sum_{i,j=1}^{n}b_{ij}\int_{\Omega}f_i \int_\eps^\infty( f_j - e^{t\Delta}f_j )t^{-1-\alpha}\dd t \dd x\\
&= \int_\eps^\infty t^{-1-\alpha}\left(
\int_\Omega\sum_{i,j=1}^n b_{ij}f_i f_j \dd x - 
\int_\Omega\sum_{i,j=1}^n b_{ij}f_i e^{t\Delta}f_j \dd x
\right)\dd t.
\end{align*}
Being $e^{t\Delta} = e^{(t/2)\Delta}e^{(t/2)\Delta}$ self-adjoint, it follows
\begin{align*}
\sum_{i,j=1}^{n}b_{ij}\int_{\Omega}f_i \Lea [f_j] \dd x &= 
\int_\eps^\infty t^{-1-\alpha}\left(
\int_\Omega\sum_{i,j=1}^n b_{ij}f_i f_j \dd x - 
\int_\Omega\sum_{i,j=1}^n b_{ij}(e^{(t/2)\Delta}f_i) (e^{(t/2)\Delta}f_j) \dd x
\right)\dd t.
\end{align*}
Let $u_i(t) = e^{t\Delta}f_i$, $i\in\{1,\ldots,n\}$. It holds
\begin{align*}
\pa_t u_i - \Delta u_i = 0\quad \mbox{on }\Omega\times (0,+\infty),\qquad
\pa_\nu u_i = 0\quad\mbox{on }\pa\Omega\times (0,+\infty),\qquad
u_i(0) = f_i\quad\mbox{on }\Omega.
\end{align*}
Therefore, being $b_{ij}$ symmetric and positive definite
\begin{align*}
\frac{\d}{\d t}\int_\Omega \sum_{i,j=1}^n b_{ij}u_i u_j \dd x = 
2\int_\Omega \sum_{i,j=1}^n b_{ij}u_i \pa_t u_j \dd x = 
2\int_\Omega \sum_{i,j=1}^n b_{ij}u_i \Delta u_j \dd x = 
-2\int_\Omega \sum_{i,j=1}^n b_{ij}\nabla u_i\cdot \nabla u_j \dd x\leq 0.
\end{align*}
It follows
\begin{align*}
\int_\Omega\sum_{i,j=1}^n b_{ij}(e^{(t/2)\Delta}f_i) (e^{(t/2)\Delta}f_j) \dd x &= \int_\Omega \sum_{i,j=1}^n b_{ij}u_i(t/2) u_j(t/2) \dd x\\ 
&\leq \int_\Omega \sum_{i,j=1}^n b_{ij}u_i(0) u_j(0) \dd x
\\&= \int_\Omega \sum_{i,j=1}^n b_{ij}f_i f_j \dd x.
\end{align*}
We infer that \eqref{sys.mono} holds.
Estimate \eqref{Le.err} is proved in \cite[Lemma 1]{MR4333988}.
Relation \eqref{Le.spec} is easily shown by exploiting the spectral decomposition of $e^{t\Delta}$. This finishes the proof of Prop.~\ref{Prop.Le}.
\end{proof}

\begin{proof}[Proof of Theorem \ref{thm:ex}]

 We shall argue by approximation by combining ideas from both \cite{MR4401974} and \cite{MR3916644}.

\textbf{Step 1.} \emph{Construction of the approximated problem.}
Let us first define, for a generic but fixed $M>0$, the truncation operator
	\begin{align*}
	T_M(x) = \min\{ \max\{x,0\}, M \}\quad x\in\R .
	\end{align*}
Let us fix $\eps, \kappa, \rho >0$ and consider the approximated problem:
\begin{align}\label{1.app}
&\pa_t u_i^{(M,\rho,\eps,\kappa)} - \kappa\Delta u_i^{(M,\rho,\eps,\kappa)}
+\kappa g_\rho[u_i^{(M,\rho,\eps,\kappa)}] \\ 
\nonumber
&\ = \Div\left(
\sum_{j=1}^n a_{ij} {T_M(u_i^{(M,\rho,\eps,\kappa)})}\nabla (-\Delta)^{-1} {\mathcal L_\eps^{(\beta + 1)/4}\mathcal L_\eps^{(\beta + 1)/4}}u_j^{(M,\rho,\eps,\kappa)}\right), && (t,x) \in (0,+\infty) \times \Omega,\\
\label{1.app.bc}
&\kappa\nu \cdot\nabla u_i^{(M,\rho,\eps,\kappa)} \\ \nonumber
& \ + \sum_{j=1}^n a_{ij}{T_M(u_i^{(M,\rho,\eps,\kappa)})}\nu\cdot\nabla (-\Delta)^{-1} {\mathcal L_\eps^{(\beta + 1)/4}\mathcal L_\eps^{(\beta + 1)/4}} u_j^{(M,\rho,\eps,\kappa)} = 0, && (t,x) \in (0,+\infty)\times \partial\Omega,\\
\label{1.app.ic}
& u_i^{(M,\rho,\eps,\kappa)}(0,\cdot) = u_{0,i}, && x \in \Omega,
\end{align}
where $\mathcal L^{(\beta+1)/4}_\eps$ is defined in \eqref{eq:approx-L} and $g_\rho[\cdot]$ in \eqref{eq:g-ap}.


The proof of the well-posedness of \eqref{1.app}--\eqref{1.app.ic} and the limits $\rho\to 0$, $\eps \to 0^+$, $\kappa\to 0$, and $M \to \infty$, proceeds in a totally similar way as the corresponding proof of \cite[Theorem 1]{MR4401974}.

\textbf{Step 2.} \emph{Local well-posedness of the approximated problem.} We prove the existence of solutions to \eqref{1.app}--\eqref{1.app.ic} for short times by applying Banach fixed point theorem. Let $R > 2\|u_0\|_{L^2(\Omega)}$, $T>0$ generic but fixed. Define the spaces
\begin{align*}
X_T &\coloneqq  C^0([0,T]; L^2(\Omega))\cap L^2((0,T); H^1(\Omega)),\\
X_{R,T} &\coloneqq  \left\{ v \in X ~~:~~ \sum_{i=1}^n(\|v_i\|_{L^\infty((0,T); L^2(\Omega))} + \|v_i\|_{L^2((0,T); H^1(\Omega))})\leq R \right\}.
\end{align*}
Consider the mapping $F : v\in X_{R,T}\mapsto u\in X_{R,T}$, where $u=(u_1,\ldots,u_n)$ is the unique solution to the linear problem
\begin{align}\label{lin.1}
&\pa_t u_i - \kappa\Delta u_i \\ \nonumber
&\quad = -\kappa g_\rho[v_i] + 
\Div\left(
\sum_{j=1}^n a_{ij} T_M(v_i)\nabla (-\Delta)^{-1} \mathcal L_\eps^{(\beta+1)/4}\mathcal L_\eps^{(\beta+1)/4}v_j\right), && (t,x) \in  (0,T)\times \Omega,\\
\label{lin.1.bc}
&\nu\cdot\nabla u_i = 0, && (t,x) \in (0,T) \times \pa\Omega,\\
\label{lin.1.ic}
& u_i(0) = u_{0,i}, && x \in \Omega.
\end{align}
Since this is just the heat equation with a forcing term belonging to $L^2((0,T); H^{1}(\Omega)')$ (as the operators $\Lea : L^2(\Omega)\to L_0^2(\Omega)$, $\pa_{x_j} (-\Delta)^{-1} : L_0^2(\Omega)\to L_0^2(\Omega)$ ($j=1,\ldots,d$) are easily bounded), we deduce that \eqref{lin.1}--\eqref{lin.1.ic} has a unique solution $u_i\in L^2((0,T); H^1(\Omega))$ such that $\pa_t u_i\in L^2((0,T); H^1(\Omega)')$ and thus (via standard results) $u_i\in C^0([0,T]; L^2(\Omega))$.
We will now show that $F$ is well defined, namely that $F(X_{R,T})\subset X_{R,T}$. Employing $u_i$ as a test function in the weak formulation of \eqref{lin.1} and then using \eqref{grho.1} and the boundedness of the operators $\Lea : L^2(\Omega)\to L_0^2(\Omega)$, $\pa_{x_j} (-\Delta)^{-1} : L_0^2(\Omega)\to L_0^2(\Omega)$ yields
\begin{align*}
\frac{1}{2}\int_{\Omega} & |u(t)|^2 \dd x - \frac{1}{2}\int_{\Omega}|u_0|^2 \dd x
+ \kappa\int_0^t\int_\Omega |\nabla u|^2 \dd x \dd \tau \\
=& 
-\kappa\int_0^t\int_\Omega \sum_{i=1}^n u_i g_\rho[v_i]\dd x \dd \tau 
-\sum_{i,j=1}^n a_{ij}\int_0^t\int_\Omega\nabla u_i\cdot T_M(v_i)\nabla (-\Delta)^{-1}\Le^{(\beta+1)/4}\Le^{(\beta+1)/4}v_j \dd x \dd \tau\\
\leq & \kappa C(\rho)\int_0^t\|u\|_{L^2(\Omega)} \|v\|_{L^2(\Omega)} \dd \tau +
C(\eps) M\int_0^t\|\nabla u\|_{L^2(\Omega)} \|v\|_{L^2(\Omega)} \dd \tau
\end{align*}
Young's inequality leads to
\begin{align*}
\frac{1}{2}\int_{\Omega} & |u(t)|^2 \dd x - \frac{1}{2}\int_{\Omega}|u_0|^2 \dd x
+ \frac{1}{2}\kappa\int_0^t\int_\Omega |\nabla u|^2 \dd x \dd \tau \\
\leq & \kappa C(\rho)\int_0^t\int_\Omega |u|^2 \dd x \dd \tau 
+ C(\rho,\kappa,\eps,M)\int_0^t\|v\|_{L^2(\Omega)}^2 \dd \tau
\end{align*}
Since $v\in X_{R,T}$ and $\|u_0\|_{L^2(\Omega)} < R/2$, via Gr\"onwall's lemma we deduce that $u\in X_{R,T}$ provided that $T$ is small enough. Therefore $F$ is well defined.

We will employ Banach's theorem to prove that $F$ has a fixed point in $X_{R,T}$ provided that $T$ is small enough. For $v, v'\in X_{R,T}$ let $u = F(v)$, $u'=F(v')$. Given \eqref{lin.1}, the difference $u-u'$ satisfies
\begin{align*}
\pa_t(u_i - u_i') - \kappa\Delta(u_i - u_i') =& 
-\kappa( g_\rho[v_i] - g_\rho[v'_i] )\\
&+\Div\left(
\sum_{j=1}^n a_{ij} (T_M(v_i) - T_M(v_i'))\nabla (-\Delta)^{-1} \Le^{(\beta+1)/4}\Le^{(\beta+1)/4}v_j\right)\\
&+\Div\left(
\sum_{j=1}^n a_{ij} T_M(v_i')\nabla (-\Delta)^{-1} \Le^{(\beta+1)/4}\Le^{(\beta+1)/4}(v_j - v_j')\right).
\end{align*}
Testing the above equation against $u_i-u_i'$ and summing in $i$ yields
\begin{align*}
\frac{1}{2}\int_\Omega & |u(t)-u'(t)|^2 \dd x + \kappa\int_0^t\int_\Omega|\nabla (u-u')|^2 \dd x \dd \tau  \\
=& -\kappa\int_0^t\int_\Omega
\sum_{i=1}^n( g_\rho[v_i] - g_\rho[v'_i] )(u_i-u_i')\dd x \dd \tau\\
&-\int_0^t\int_\Omega \sum_{i,j=1}^n a_{ij}\nabla (u_i-u_i')\cdot
(T_M(v_i) - T_M(v_i'))\nabla (-\Delta)^{-1} \Le^{(\beta+1)/4}\Le^{(\beta+1)/4}v_j \dd x \dd \tau\\
&-\int_0^t\int_\Omega \sum_{i,j=1}^n a_{ij}\nabla (u_i-u_i')\cdot
T_M(v_i')\nabla (-\Delta)^{-1} \Le^{(\beta+1)/4}\Le^{(\beta+1)/4}(v_j - v_j')
\dd x \dd \tau
\end{align*}
Applying Cauchy--Schwartz and H\"older's inequalities leads to
\begin{align*}
\int_\Omega |u(t)& -u'(t)|^2 \dd x + \int_0^t\int_\Omega|\nabla (u-u')|^2 \dd x \dd \tau \\
\lesssim_\kappa & \int_0^t \sum_{i=1}^n\|g_\rho[v_i] - g_\rho[v'_i]\|_{L^2(\Omega)}
\|u_i - u_i'\|_{L^2(\Omega)}\dd x \dd \tau\\
&+\sum_{i,j=1}^n \int_0^t\|\nabla (u_i - u_i')\|_{L^2(\Omega)}
\|T_M(v_i)-T_M(v_i')\|_{L^2(\Omega)}
\|\nabla (-\Delta)^{-1}\Le^{(\beta+1)/4}\Le^{(\beta+1)/4}v_j\|_{L^\infty(\Omega)}\dd \tau\\
&+\sum_{i,j=1}^n\int_0^t\|\nabla (u_i - u_i')\|_{L^2(\Omega)}
\|T_M(v_i')\|_{L^\infty(\Omega)}
\|\nabla (-\Delta)^{-1} \Le^{(\beta+1)/4}\Le^{(\beta+1)/4}(v_j - v_j')\|_{L^2(\Omega)}\dd \tau
\end{align*}
From \eqref{grho.1} and the definition of $T_M$ it follows
\begin{align*}
\int_\Omega |u(t)& -u'(t)|^2 \dd x + \int_0^t\int_\Omega|\nabla (u-u')|^2 \dd x \dd \tau \\
\lesssim_{\kappa,\rho,M} & \int_0^t \|u - u'\|_{L^2(\Omega)}^2 \dd x \dd \tau
+ \int_0^t \|v - v'\|_{L^2(\Omega)}^2 \dd x \dd \tau\\
&+\sum_{i,j=1}^n \int_0^t\|T_M(v_i)-T_M(v_i')\|_{L^2(\Omega)}^2
\|\nabla (-\Delta)^{-1}\Le^{(\beta+1)/4}\Le^{(\beta+1)/4}v_j\|_{L^\infty(\Omega)}^2 \dd \tau\\
&+\sum_{j=1}^n\int_0^t
\|\nabla (-\Delta)^{-1} \Le^{(\beta+1)/4}\Le^{(\beta+1)/4}(v_j - v_j')\|_{L^2(\Omega)}^2 \dd \tau.
\end{align*}
From \eqref{TM.1}, it follows
\begin{align*}
\sum_{i=1}^n\|T_M(v_i)-T_M(v_i')\|_{L^2(\Omega)}^2 \lesssim_M
\|v-v'\|_{L^2(\Omega)}^{2},
\end{align*}
and Gagliardo--Nirenberg--Sobolev's interpolation inequality 
(NB: here we need the hypothesis $d < 4$) lead to
\begin{align*}
\|\nabla (-\Delta)^{-1}\Le^{(\beta+1)/4}\Le^{(\beta+1)/4}v_j\|_{L^\infty(\Omega)} &\lesssim
\|\nabla (-\Delta)^{-1}\Le^{(\beta+1)/4}\Le^{(\beta+1)/4}v_j\|_{W^{1,p}(\Omega)}
\\ &\lesssim
\| \Le^{(\beta+1)/4}\Le^{(\beta+1)/4}v_j \|_{L^p(\Omega)}
\\ &\lesssim_\eps
\|v_j\|_{L^p(\Omega)}\\
&\lesssim_\eps \|v_j\|_{L^2(\Omega)}^{1-\theta}\|v_j\|_{H^1(\Omega)}^\theta
\end{align*}
but being $v\in X_{R,T}$ we deduce
\begin{align*}
\|\nabla (-\Delta)^{-1}\Le^{(\beta+1)/4}\Le^{(\beta+1)/4}v_j\|_{L^\infty(\Omega)}\lesssim_{\eps,R}
\|v_j\|_{H^1(\Omega)}^\theta,\quad \theta=\theta(p) \in (0,1).
\end{align*}
It also holds
\begin{align*}
\|\nabla (-\Delta)^{-1} \Le^{(\beta+1)/4}\Le^{(\beta+1)/4}(v_j - v_j')\|_{L^2(\Omega)}\lesssim_\eps \|v-v'\|_{L^2(\Omega)}.
\end{align*}
We infer, via H\"older's inequality and the fact that $v\in X_{R,T}$, that 
\begin{align*}
\sum_{i,j=1}^n \int_0^t &\|T_M(v_i)-T_M(v_i')\|_{L^2(\Omega)}^2
\|\nabla (-\Delta)^{-1}\Le^{(\beta+1)/4}\Le^{(\beta+1)/4}v_j\|_{L^\infty(\Omega)}^2 \dd \tau\\
&\lesssim_{M,\eps,R}\int_0^t \|v\|_{H^1(\Omega)}^{2\theta}\|v-v'\|_{L^2(\Omega)}^{2} \dd \tau\\
&\lesssim_{M,\eps,R}t^{1-\theta} \|v-v'\|_{X_T}
\end{align*}
and so
\begin{align*}
\|u(t) &-u'(t)\|_{L^2(\Omega)}^2 + \int_0^t\|\nabla (u-u')\|_{L^2(\Omega)}\dd \tau\\ 
&\lesssim_{\kappa,\rho,M,\eps,R}
\int_0^t \|u-u'\|_{L^2(\Omega)}^2 \dd \tau +
t^{1-\theta} \|v-v'\|_{X_T}^2 .
\end{align*}
Gr\"onwall's lemma and trivial estimates lead to
\begin{align*}
\sup_{t\in [0,T]}\|u(t) &-u'(t)\|_{L^2(\Omega)}^2 + \int_0^T\|\nabla (u-u')\|_{L^2(\Omega)}\dd \tau
\lesssim_{\kappa,\rho,M,\eps,R}
e^T T^{1-\theta} \|v-v'\|_{X}^2 .
\end{align*}
Since $0<\theta<1$, choosing $T>0$ small enough yields that $F$ is a contraction on $X_{R,T}$ and therefore has a unique fixed point $u\in X_{R,T}$ which is a weak solution to \eqref{1.app}--\eqref{1.app.ic}.

\textbf{Step 3.} \emph{Uniform bounds and global well-posedness for the approximated problem.} 
We wish now to derive some bounds for the solution $u$ to \eqref{1.app}--\eqref{1.app.ic} that are uniform in the approximation parameters. For $\eta>0$ arbitrary, define
\begin{align}\label{H.M.eta}
H_{M,\eta}[u]\equiv \int_\Omega \sum_{i=1}^n\pi_i h_{M,\eta}(u_i)\dd x,\quad 
h_{M,\eta}(v)\equiv \int_0^v\int_1^s ( \eta + T_M(\sigma) )^{-1}\dd\sigma \dd s\quad \text{ for all } v\in\R_+ ,\\
\label{H.M}
H_{M}[u]\equiv \int_\Omega \sum_{i=1}^n\pi_i h_{M}(u_i)\dd x,\quad 
h_{M}(v)\equiv \int_0^v\int_1^s T_M(\sigma)^{-1}\dd\sigma \dd s\quad \text{ for all } v\in\R_+ .
\end{align}
It holds that $h_{M,\eta}'(u_i^{(M,\rho,\eps,\kappa)})\in L^2((0,T); H^1(\Omega))$ since $u_i^{(M,\rho,\eps,\kappa)}\in L^2((0,T); H^1(\Omega))$, thus we can employ $\pi_i h_{M,\eta}'(u_i^{(M,\rho,\eps,\kappa)})$ as a test function in \eqref{1.app} and obtain (via summing the resulting equations in $i$)
\begin{align*}
H_{M,\eta}&[\uMrek(t)] 
+ \kappa\sum_{i=1}^n \int_0^t\int_\Omega\frac{\pi_i |\nabla \uMrek_i|^2}{
\eta + T_M(\uMrek_i)}\dd x \dd\tau\\
&+ \kappa \int_0^t\int_\Omega\sum_{i=1}^n\pi_i g_\rho[\uMrek_i] h_{M,\eta}'(\uMrek_i)\dd x \dd\tau\\
&+\sum_{i,j=1}a_{ij}\pi_i \int_0^t\int_\Omega\frac{T_M(\uMrek_i)}{\eta+T_M(\uMrek_i)}
\nabla \uMrek_i\cdot\nabla (-\Delta)^{-1}\Le^{(\beta+1)/4}\Le^{(\beta+1)/4}\uMrek_j \dd x \dd\tau\\
=& H_{M,\eta}[u_0] 
\end{align*}
and so
\begin{align}\label{H.m.eta.1}
& H_{M,\eta}[\uMrek(t)] 
+ \kappa\sum_{i=1}^n \int_0^t\int_\Omega\frac{\pi_i |\nabla \uMrek_i|^2}{
	\eta + T_M(\uMrek_i)}\dd x d\tau\\ \nonumber
&\quad +\sum_{i,j=1}a_{ij}\pi_i \int_0^t\int_\Omega
\uMrek_i \Le^{(\beta+1)/4}\Le^{(\beta+1)/4}\uMrek_j \dd x \dd\tau\\ \nonumber
&= H_{M,\eta}[u_0] +
\sum_{i,j=1}a_{ij}\pi_i \int_0^t\int_\Omega\frac{\eta\nabla \uMrek_i}{\eta+T_M(\uMrek_i)}
\cdot\nabla (-\Delta)^{-1}\Le^{(\beta+1)/4}\Le^{(\beta+1)/4}\uMrek_j \dd x \dd\tau\\ \nonumber
&\quad -\kappa \int_0^t\int_\Omega\sum_{i=1}^n\pi_i g_\rho[\uMrek_i] h_{M,\eta}'(\uMrek_i)\dd x \dd\tau
\end{align}
If we define 
$$g_{M,\eta}(s) = \int_0^s \frac{\eta}{\eta+T_M(\sigma)}\dd\sigma $$
then the second integral on the right-hand side of \eqref{H.m.eta.1} can be rewritten as
\begin{align*}
&\sum_{i,j=1}a_{ij}\pi_i \int_0^t\int_\Omega
\nabla g_{M,\eta}(\uMrek_i)\cdot\nabla (-\Delta)^{-1}\Le^{(\beta+1)/4}\Le^{(\beta+1)/4}\uMrek_j \dd x \dd \tau\\
&= \sum_{i,j=1}a_{ij}\pi_i \int_0^t\int_\Omega
g_{M,\eta}(\uMrek_i)\Le^{(\beta+1)/4}\Le^{(\beta+1)/4}\uMrek_j \dd x \dd\tau = 
\sum_{i,j=1}a_{ij}\pi_i(J_{i,\eta,1} + J_{i,\eta,2}),
\end{align*}
with 
\begin{align*}
J_{i,\eta,1}\equiv \int_{ \{ \uMrek_i\leq\sqrt{\eta} \} }
g_{M,\eta}(\uMrek_i)\Le^{(\beta+1)/4}\Le^{(\beta+1)/4}\uMrek_j \dd x \dd \tau,\\
J_{i,\eta,2}\equiv \int_{ \{ \uMrek_i > \sqrt{\eta} \} }
g_{M,\eta}(\uMrek_i)\Le^{(\beta+1)/4}\Le^{(\beta+1)/4}\uMrek_j \dd x \dd \tau.
\end{align*}
It follows
\begin{align*}
|J_{i,\eta,1}| + |J_{i,\eta,2}|\leq \sqrt{\eta}C(M,\rho,\eps,\kappa,T)\to 0\quad\mbox{as }\eta\to 0.
\end{align*}
On the other hand, given \eqref{eq:g-ap} the third integral on the right-hand side of \eqref{H.m.eta.1} reads as
\begin{align*}
-\kappa &\int_0^t\int_\Omega\sum_{i=1}^n\pi_i g_\rho[\uMrek_i] h_{M,\eta}'(\uMrek_i)\dd x \dd \tau \\ 
&= 
-\kappa \int_0^t\int_\Omega\sum_{i=1}^n\pi_i\frac{(\uMrek_i)^2}{1+\rho |\uMrek_i|}h_{M,\eta}'(\uMrek_i) \dd x \dd \tau \\
&\quad +\kappa \int_0^t\int_\Omega\sum_{i=1}^n\pi_i
\fint_\Omega\frac{(\uMrek_i)^2}{1+\rho |\uMrek_i|} \dd y h_{M,\eta}'(\uMrek_i) \dd x \dd \tau 
\end{align*}
and given \eqref{H.M.eta} it follows
\begin{align}\label{gh}
	-\kappa &\int_0^t\int_\Omega\sum_{i=1}^n\pi_i g_\rho[\uMrek_i] h_{M,\eta}'(\uMrek_i)\dd x \dd \tau \\ \nonumber
	&=
	-\kappa \sum_{i=1}^n\pi_i\int_{ \{ (x,\tau):~\uMrek_i(x,\tau)\leq 1 \} }\frac{(\uMrek_i)^2}{1+\rho |\uMrek_i|}h_{M,\eta}'(\uMrek_i) \dd x \dd \tau \\ \nonumber
	&\quad 
	-\kappa \sum_{i=1}^n\pi_i\int_{ \{ (x,\tau):~\uMrek_i(x,\tau)> 1 \} }\frac{(\uMrek_i)^2}{1+\rho |\uMrek_i|}h_{M,\eta}'(\uMrek_i) \dd x \dd \tau \\ \nonumber
	&\quad 
	+\kappa \sum_{i=1}^n\pi_i\int_{ \{ (x,\tau):~\uMrek_i(x,\tau)\leq 1 \} }
	\fint_\Omega\frac{(\uMrek_i(y,\tau))^2}{1+\rho |\uMrek_i(y,\tau)|} \dd y~ h_{M,\eta}'(\uMrek_i(x,\tau)) \dd x \dd \tau\\ \nonumber
	&\quad +\kappa \sum_{i=1}^n\pi_i\int_{ \{ (x,\tau):~\uMrek_i(x,\tau)> 1 \} }
	\fint_\Omega\frac{(\uMrek_i(y,\tau))^2}{1+\rho |\uMrek_i(y,\tau)|} \dd y~ h_{M,\eta}'(\uMrek_i(x,\tau)) \dd x \dd \tau .
\end{align}
From \eqref{gh} we obtain
\begin{align*}
	-\kappa &\int_0^t\int_\Omega\sum_{i=1}^n\pi_i g_\rho[\uMrek_i] h_{M,\eta}'(\uMrek_i)\dd x \dd \tau \\ \nonumber
	&\leq C -\kappa \sum_{i=1}^n\pi_i\int_{ \{ (x,\tau):~\uMrek_i(x,\tau)> 1 \} }\frac{(\uMrek_i)^2}{1+\rho |\uMrek_i|}h_{M,\eta}'(\uMrek_i) \dd x \dd \tau \\ \nonumber
	&\quad +\kappa \sum_{i=1}^n\pi_i\int_{ \{ (x,\tau):~\uMrek_i(x,\tau)> 1 \} }
	\fint_\Omega\frac{(\uMrek_i(y,\tau))^2}{1+\rho |\uMrek_i(y,\tau)|} \dd y~ h_{M,\eta}'(\uMrek_i(x,\tau)) \dd x \dd \tau \\
	&\leq C - \kappa \sum_{i=1}^n\pi_i
	\int_0^t\int_{\Omega}\frac{(\uMrek_i)^2}{1+\rho |\uMrek_i|}(\log(\uMrek_i))_+ \dd x \dd \tau\\
	&\quad + \kappa |\Omega|^{-1} \sum_{i=1}^n\pi_i\int_0^t
	\left(\int_\Omega \frac{(\uMrek_i)^2}{1+\rho |\uMrek_i|} \dd y\right)
	\left(\int_\Omega [
	( \log (\uMrek_i) )_+ + M^{-1}(\uMrek_i - M)_+ ] \dd x
	\right)\dd\tau \\
	&\leq C - \kappa\sum_{i=1}^n\pi_i\int_0^t \left(\fint_\Omega (\log(\uMrek_i))_+ dy\right)
	\int_\Omega\frac{(\uMrek_i)^2}{1+\rho |\uMrek_i|}\left(
	\frac{( \log (\uMrek_i) )_+}{ \fint_\Omega (\log(\uMrek_i))_+ \dd y } - 1
	\right)\dd x\\
	&\quad + \frac{\kappa}{|\Omega|M}\sum_{i=1}^n\pi_i
	\int_0^t
	\left(\int_\Omega \frac{(\uMrek_i)^2}{1+\rho |\uMrek_i|} \dd y\right)
	\int_\Omega(\uMrek_i - M)_+ \dd x\\
	&\leq C - \kappa\sum_{i=1}^n\pi_i\int_0^t \left(\fint_\Omega (\log(\uMrek_i))_+ dy\right)
	\int_\Omega\frac{(\uMrek_i)^2}{1+\rho |\uMrek_i|}\left(
	\frac{( \log (\uMrek_i) )_+}{ \fint_\Omega (\log(\uMrek_i))_+ \dd y } - 1
	\right)_+ \dd x\\
	&\quad - \kappa\sum_{i=1}^n\pi_i\int_0^t \left(\fint_\Omega (\log(\uMrek_i))_+ dy\right)
	\int_\Omega\frac{(\uMrek_i)^2}{1+\rho |\uMrek_i|}\left(
	\frac{( \log (\uMrek_i) )_+}{ \fint_\Omega (\log(\uMrek_i))_+ \dd y } - 1
	\right)_- \dd x\\
	&\quad + \frac{\kappa}{M}\sum_{i=1}^n\pi_i
	\int_0^t
	\left(\int_\Omega (\uMrek_i)^2 \dd y\right)
	\fint_\Omega (\uMrek_i - M)_+ \dd x.
\end{align*}
However, since
$\fint_\Omega (\log(\uMrek_i))_+ \dd y\leq \fint_\Omega\uMrek_i \dd x = \fint_\Omega u_{0,i}dx<\infty$
and $s\mapsto s^2(1+\rho s)^{-1}$ is nondecreasing we deduce
\begin{align*}
	- \kappa\sum_{i=1}^n\pi_i &\int_0^t \left(\fint_\Omega (\log(\uMrek_i))_+ dy\right)
	\int_\Omega\frac{(\uMrek_i)^2}{1+\rho |\uMrek_i|}\left(
	\frac{( \log (\uMrek_i) )_+}{ \fint_\Omega (\log(\uMrek_i))_+ dy } - 1
	\right)_- \dd x
	\leq C\kappa.
\end{align*}
Finally, since the operator $\Le^{(\beta+1)/4}$ is self-adjoint and the matrix $(\pi_i a_{ij})_{i,j=1,\dots,n}$ is positive definite, we infer
	\begin{align*}
		\sum_{i,j=1}a_{ij}\pi_i \int_0^t\int_\Omega \uMrek_i
		\Le^{(\beta+1)/4}\Le^{(\beta+1)/4}\uMrek_j \dd x \dd \tau\geq 
		c_0\int_0^t\int_\Omega |\Le^{(\beta+1)/4}\uMrek|^2 \dd x \dd \tau,
	\end{align*} %
where $c_0$ is the smallest eigenvalue of the matrix $(\pi_i a_{ij})_{i,j=1,\dots,n}$.
Therefore taking the limit (inferior) $\eta\to 0$ on both sides of \eqref{H.m.eta.1} and applying Fatou's Lemma leads to
\begin{align*}
	&H_M[\uMrek(t)] + \kappa\sum_{i=1}^n \int_0^t\int_\Omega\frac{\pi_i |\nabla \uMrek_i|^2}{T_M(\uMrek_i)}\dd x \dd \tau\\
	&\quad +c_0\int_0^t\int_\Omega |\Le^{(\beta+1)/4}\uMrek|^2 \dd x \dd \tau\\ \nonumber
	&\quad +\kappa\sum_{i=1}^n\pi_i\int_0^t \left(\fint_\Omega (\log(\uMrek_i))_+ dy\right)
	\int_\Omega\frac{(\uMrek_i)^2}{1+\rho |\uMrek_i|}\left(
	\frac{( \log (\uMrek_i) )_+}{ \fint_\Omega (\log(\uMrek_i))_+ \dd y } - 1
	\right)_+ \dd x\\ \nonumber
&\leq C + \frac{\kappa}{M}\sum_{i=1}^n\pi_i
\int_0^t \left(\int_\Omega (\uMrek_i)^2 dy\right)
\fint_\Omega (\uMrek_i - M)_+ \dd x.
\end{align*}
However, since 
\begin{align*}
h_M(u)\geq -c_1 + \frac{(u-1)^2}{2M}\quad\text{ for all } u\geq 0
\end{align*}
we deduce
\begin{align}\label{ei.M.1}
	H_M[& \uMrek(t)] + \kappa\sum_{i=1}^n \int_0^t\int_\Omega\frac{\pi_i |\nabla \uMrek_i|^2}{T_M(\uMrek_i)}\dd x \dd \tau\\ \nonumber
	&+c_0\int_0^t\int_\Omega |\Le^{(\beta+1)/4}\uMrek|^2 \dd x \dd \tau\\ \nonumber
	&+\kappa\sum_{i=1}^n\pi_i\int_0^t 
	\int_\Omega\frac{(\uMrek_i)^2}{1+\rho |\uMrek_i|}\left(
	( \log (\uMrek_i) )_+ - \fint_\Omega (\log(\uMrek_i))_+ \dd y
	\right)_+ \dd x\\ \nonumber
	\leq & C(1+\kappa t) + C\kappa 
	\int_0^t H_M[\uMrek]
	\sum_{i=1}^n\fint_\Omega (\uMrek_i - M)_+ \dd x.
\end{align}
Estimate \eqref{ei.M.1} and Gr\"onwall's lemma allow us to conclude that the solution $\uMrek$ to \eqref{1.app}--\eqref{1.app.ic} exists globally in time.

%
%
%
%
%
%
%
%
%
%
%
%
%
%
%
%
%
%
%
%
\textbf{Step 4.} \emph{Limit process.}
We will carry out the limits $M \to +\infty$, $\rho \to 0^+$, $\eps \to 0^+$, $\kappa \to 0^+$ in this order. The fundamental tool for these steps will be \eqref{ei.M.1}. 

\textbf{Step 4a.} \emph{Limit $M \to +\infty$.}
Since $T_M(s)\leq s$, from \eqref{ei.M.1}, Gr\"onwall's lemma and mass conservation it follows
\begin{align}\label{limM.est.1}
\|\uMrek_i\|_{L^\infty((0,T); L^1(\Omega))} + 
\|\nabla \sqrt{\uMrek_i}\|_{L^2((0,T); L^2(\Omega))}\leq C.
\end{align}
Gagliardo--Nirenberg--Sobolev's inequality yields
\begin{align}\label{limM.est.1b}
\|\uMrek_i\|_{L^2((0,T); L^{3/2}(\Omega))}\leq C.
\end{align}
On the other hand from \eqref{ei.M.1} and the easy lower bound
\begin{align}\label{hM.lb}
h_M(u)\geq C + u\log u - u + 1\quad u\geq 1
\end{align}
we deduce
\begin{align}\label{limM.est.2}
\|\uMrek_i\log\uMrek_i\|_{L^\infty((0,T); L^1(\Omega))}\leq C.
\end{align}
It is straightforward from \eqref{limM.est.1}, \eqref{limM.est.2}, \eqref{1.app} to derive an estimate for the time derivative of $\uMrek$:
\begin{align}\label{limM.est.3}
\|\pa_t\uMrek_i\|_{L^1((0,T); W^{2,p}(\Omega)')}\leq C\quad\mbox{for some }p>3.
\end{align}
Via Aubin--Lions' lemma we deduce that, up to subsequences
\begin{align}\label{limM.cnv}
\uMrek_i\to \urek_i\quad\mbox{strongly in }L^1(Q_T).
\end{align}
Also, from \eqref{limM.est.1b} and Fatou's Lemma it follows that $\urek_i\in L^2((0,T); L^{3/2}(\Omega))$.
We will now derive a stronger convergence property.
Let $R>2$ arbitrary. Consider
\begin{align*}
\| &\uMrek_i - \urek_i\|_{L^2((0,T); L^{3/2}(\Omega))}\\
&\leq 
\|T_R(\uMrek_i) - \uMrek_i\|_{L^2((0,T); L^{3/2}(\Omega))} +
\|T_R(\urek_i) - \urek_i\|_{L^2((0,T); L^{3/2}(\Omega))} \\
&\qquad 
+\|T_R(\uMrek_i) - T_R(\urek_i)\|_{L^2((0,T); L^{3/2}(\Omega))}
\end{align*}
Clearly the third term on the right-hand side of the above inequality tends to 0 as $M\to\infty$ thanks to \eqref{limM.cnv}. This fact and Fatous Lemma yield
\begin{align}\label{limM.1}
\limsup_{M\to\infty}\| &\uMrek_i - \urek_i\|_{L^2((0,T); L^{3/2}(\Omega))}
\leq 2\limsup_{M\to\infty}\|T_R(\uMrek_i) - \uMrek_i\|_{L^2((0,T); L^{3/2}(\Omega))}.
\end{align}
However, $L^p$ interpolation leads to
\begin{align*}
\|& T_R(\uMrek_i) - \uMrek_i\|_{L^2((0,T); L^{3/2}(\Omega))}\\
&\leq 
\|T_R(\uMrek_i) - \uMrek_i\|_{L^\infty((0,T); L^{1}(\Omega))}^{1/2}
\|T_R(\uMrek_i) - \uMrek_i\|_{L^1((0,T); L^{3}(\Omega))}^{1/2}\\
&\leq 
\|T_R(\uMrek_i) - \uMrek_i\|_{L^\infty((0,T); L^{1}(\Omega))}^{1/2}
\|\sqrt{\uMrek_i}\|_{L^2((0,T); L^{6}(\Omega))}
\end{align*}
and Sobolev's inequality yields
\begin{align*}
\|& T_R(\uMrek_i) - \uMrek_i\|_{L^2((0,T); L^{3/2}(\Omega))}\\
&\leq C
\|T_R(\uMrek_i) - \uMrek_i\|_{L^\infty((0,T); L^{1}(\Omega))}^{1/2}
\|\sqrt{\uMrek_i}\|_{L^2((0,T); H^{1}(\Omega))}.
\end{align*}
From \eqref{limM.est.1} it follows
\begin{align*}
\|& T_R(\uMrek_i) - \uMrek_i\|_{L^2((0,T); L^{3/2}(\Omega))}^2\\
&\leq C
\|T_R(\uMrek_i) - \uMrek_i\|_{L^\infty((0,T); L^{1}(\Omega))}\\
&\leq \|\uMrek_i \chf{ \{ \uMrek_i > R \} } \|_{L^\infty((0,T); L^{1}(\Omega))}\\
&\leq \frac{1}{\log R}\|\uMrek_i \log\uMrek_i \|_{L^\infty((0,T); L^{1}(\Omega))}
\end{align*}
so from \eqref{limM.est.2} we obtain
\begin{align*}
\|& T_R(\uMrek_i) - \uMrek_i\|_{L^2((0,T); L^{3/2}(\Omega))}^2
\leq \frac{C}{\log R}.
\end{align*}
From the above inequality and \eqref{limM.1} we deduce
\begin{align*}
\limsup_{M\to\infty}\| &\uMrek_i - \urek_i\|_{L^2((0,T); L^{3/2}(\Omega))}
\leq \frac{C}{\sqrt{\log R}},
\end{align*}
which implies (since the left-hand side does not depend on $R$) that
\begin{align}\label{limM.cnv.2}
\uMrek_i\to \urek_i\quad\mbox{strongly in }L^2((0,T); L^{3/2}(\Omega)).
\end{align}
It follows
\begin{align*}
&\nabla\uMrek_i = 2\sqrt{\uMrek_i}\nabla\sqrt{\uMrek_i}\rightharpoonup
2\sqrt{\urek_i}\nabla\sqrt{\urek_i} = \nabla\urek_i\quad 
\mbox{weakly in }L^{6/5}(Q_T),\\
&g_\rho[\uMrek_i]\to g_\rho[\urek_i]\quad\mbox{weakly in }L^2((0,T); L^{3/2}(\Omega)),\\
&T_M(u_i^{(M,\rho,\eps,\kappa)})\nabla (-\Delta)^{-1} \Le^{(\beta+1)/4}\Le^{(\beta+1)/4}u_j^{(M,\rho,\eps,\kappa)}\\
&\qquad\qquad\rightharpoonup
u_i^{(\rho,\eps,\kappa)}\nabla (-\Delta)^{-1} \Le^{(\beta+1)/4}\Le^{(\beta+1)/4}u_j^{(\rho,\eps,\kappa)}\quad\mbox{weakly in }L^1(Q_T),
\end{align*}
where the last convergence follows from \eqref{limM.cnv.2}, the fact that the operators $\mathcal{L}_\eps^{(\beta+1)/2} : L^2(Q_T)\to L^2(Q_T)$,
$\nabla (-\Delta)^{-1} : L^2((0,T); L^{3/2}(\Omega))\to L^2((0,T); W^{1,3/2}(\Omega))$ are bounded, and Sobolev's embedding $W^{1,3/2}(\Omega)\hookrightarrow L^3(\Omega)$.
This allows us to take the limit $M\to\infty$ in \eqref{1.app}--\eqref{1.app.ic} and obtain a solution $\urek$ to
\begin{align}\label{2.app}
	\pa_t \urek_i &- \kappa\Delta \urek_i
	+\kappa g_\rho[\urek_i] \\ 
	\nonumber
	&= \Div\left(
	\sum_{j=1}^n a_{ij} \urek_i\nabla (-\Delta)^{-1} \Le^{(\beta+1)/4}\Le^{(\beta+1)/4} \urek_j\right), && (t,x) \in (0,+\infty) \times \Omega,\\
	\label{2.app.bc}
	\kappa\nu\cdot\nabla \urek_i &+ \sum_{j=1}^n a_{ij}\urek_i\nu\cdot\nabla (-\Delta)^{-1} \Le^{(\beta+1)/4}\Le^{(\beta+1)/4}\urek_j = 0, && (t,x) \in (0,+\infty)\times \partial\Omega,\\
	\label{2.app.ic}
	\urek_i(0,\cdot) &= u_{0,i}, && x \in \Omega.
\end{align}
Also, taking the limit $M\to\infty$ in \eqref{ei.M.1} and applying Fatou's Lemma leads to
\begin{align}\label{ei.rho.1}
	H[& \urek(t)] + 4\kappa\sum_{i=1}^n \int_0^t\int_\Omega
	\pi_i |\nabla \sqrt{\urek_i}|^2 \dd x \dd \tau\\ \nonumber
	&+c_0\int_0^t\int_\Omega |\Le^{(\beta+1)/4}\urek|^2 \dd x \dd \tau\\ \nonumber
	&+\kappa\sum_{i=1}^n\pi_i\int_0^t 
	\int_\Omega\frac{(\urek_i)^2}{1+\rho |\urek_i|}\left(
	( \log (\urek_i) )_+ - \fint_\Omega (\log(\urek_i))_+ \dd y
	\right)_+ \dd x \dd \tau\\ \nonumber
	\leq & C(1+\kappa t) .
\end{align}

\textbf{Step 4b.} \emph{Limit $\rho \to 0^+$.}
Estimates \eqref{limM.est.1}, \eqref{limM.est.1b}, \eqref{limM.est.2}, \eqref{limM.est.3} still hold, and by arguing in the same way as in the previous step one can prove
\begin{align}\label{limrho.cnv}
	\urek_i\to \uek_i\quad\mbox{strongly in }L^2((0,T); L^{3/2}(\Omega)).
\end{align}
The only thing we need to do is to study the convergence of the term $g_\rho[\urek_i]$. From \eqref{ei.rho.1} we easily deduce that 
$g_\rho[\urek_i]$ is bounded in $L^1(Q_T)$ and that
\begin{align*}
C(1+\kappa T) &\geq \int_0^T
\int_\Omega\frac{(\urek_i)^2}{1+\rho |\urek_i|}\left(
( \log (\urek_i) )_+ - \fint_\Omega (\log(\urek_i))_+ \dd y
\right)_+ \dd x \dd t\\
&\geq \int_0^T
\int_\Omega\frac{(\urek_i)^2}{1+\rho |\urek_i|}\left(
( \log (\urek_i) )_+ - m
\right)_+ \dd x \dd t,
\end{align*}
where $m = \sum_{j=1}^n \fint_\Omega u_{0,j}\, \mathrm d x$ (mass conservation).
Thus, we obtain
\begin{align*}
C(1+\kappa T) &\geq (\log R - m)
\int_{ \{ \urek_i > R \} }\frac{(\urek_i)^2}{1+\rho |\urek_i|}\, \mathrm d x \, \mathrm dt,
\end{align*}
which easily implies, together with \eqref{limrho.cnv}, the strong convergence
\begin{align*}
\frac{(\urek_i)^2}{1+\rho |\urek_i|}\to (\uek_i)^2\quad\mbox{strongly in }L^1(Q_T),
\end{align*}
which via the definition \eqref{eq:g-ap} of $g_\rho$ yields
\begin{align*}
g_\rho[\urek_i]\to g[\uek_i] \quad\mbox{strongly in }L^1(Q_T).
\end{align*}
with
\begin{align}\label{def.g}
g[u]\equiv u^2 - \fint_\Omega u^2 \, \mathrm d x.
\end{align}
Putting together the arguments in the previous step and the statement above we can take the limit $\rho\to 0$ in \eqref{2.app}--\eqref{2.app.ic} and obtain
\begin{align}\label{3.app}
	\pa_t \uek_i &- \kappa\Delta \uek_i
	+\kappa g[\uek_i] \\ 
	\nonumber
	&= \Div\left(
	\sum_{j=1}^n a_{ij} \uek_i\nabla (-\Delta)^{-1} \Le^{(\beta+1)/4}\Le^{(\beta+1)/4} \uek_j\right), && (t,x) \in (0,+\infty) \times \Omega,\\
	\label{3.app.bc}
	\kappa\nu\cdot\nabla \uek_i &+ \sum_{j=1}^n a_{ij}\uek_i\nu\cdot\nabla (-\Delta)^{-1} \Le^{(\beta+1)/4}\Le^{(\beta+1)/4} \uek_j = 0, && (t,x) \in (0,+\infty)\times \partial\Omega,\\
	\label{3.app.ic}
	\uek_i(0,\cdot) &= u_{0,i}, && x \in \Omega.
\end{align}
with $g$ as in \eqref{def.g}.
Also, taking the limit $\rho\to 0$ in \eqref{ei.rho.1} and applying Fatou's Lemma leads to
\begin{align}\label{ei.eps.1}
	H[& \uek(t)] + 4\kappa\sum_{i=1}^n \int_0^t\int_\Omega
	\pi_i |\nabla \sqrt{\uek_i}|^2 \dd x \dd\tau\\ \nonumber
	&+c_0\int_0^t\int_\Omega |\Le^{(\beta+1)/4}\uek|^2 \dd x \dd\tau\\ \nonumber
	&+\kappa\sum_{i=1}^n\pi_i\int_0^t 
	\int_\Omega (\uek_i)^2\left(
	( \log (\uek_i) )_+ - \fint_\Omega (\log(\uek_i))_+ \dd y
	\right)_+ \dd x \dd \tau\\ \nonumber
	\leq & C(1+\kappa t) .
\end{align}

\textbf{Step 4c.} \emph{Limit $\eps \to 0^+$.}
By arguing like in the previous steps we can deduce the strong convergence of $\uek\to \uk$. Given the $L^1(Q_T)$ bound for $(\uek)_i^2\log \uek_i$ from \eqref{ei.eps.1}, we deduce
\begin{align}\label{cnv.uek}
\uek_i\to \uk_i \quad\mbox{strongly in }L^2(Q_T).
\end{align}
Furthermore, from \eqref{ei.eps.1} we also deduce 
\begin{align}\label{cnv.uek.2}
\uek_i\rightharpoonup \uk_i \quad\mbox{weakly in }L^{4/3}((0,T); W^{1,4/3}(\Omega)),
\end{align}
given that $\nabla\sqrt{\uek_i}$ and $\sqrt{\uek_i}$ are bounded in $L^2(Q_T)$ and $L^4(Q_T)$, respectively.
Also from \eqref{ei.eps.1} we obtain
\begin{align*}
\Le^{(\beta+1)/4}\uek_i\rightharpoonup \Xi_i\quad\mbox{weakly in }L^2(Q_T).
\end{align*}
Therefore, given an arbitrary $\psi\in C^\infty_c(Q_T)$ it follows
\begin{align*}
\int_{Q_T}\psi (-\Delta)^{-1}\Le^{(\beta+1)/4}\uek_i \dd x \dd t \to 
\int_{Q_T}\psi (-\Delta)^{-1}\Xi_i \dd x \dd t.
\end{align*}
However,
\begin{align*}
\int_{Q_T} &\psi (-\Delta)^{-1}\Le^{(\beta+1)/4}\uek_i \dd x \dd t \\
&= 
\int_{Q_T}\psi ((-\Delta)^{-1}\Le^{(\beta+1)/4}\uek_i - (-\Delta)^{(\beta+1)/4-1}\uek ) \dd x \dd t
+ \int_{Q_T} \psi (-\Delta)^{(\beta+1)/4-1}( \uek_i - \uek_i )\dd x \dd t\\
&\qquad + \int_{Q_T} \psi (-\Delta)^{(\beta+1)/4-1}\uek_i \dd x \dd t.
\end{align*}
The first and second terms on the right-hand side of the above identity tend to zero because of \eqref{Le.err} and \eqref{cnv.uek}. We deduce from the arbitrarity of $\psi$ that $(-\Delta)^{-1}\Xi_i = (-\Delta)^{(\beta+1)/4-1}\uek_i$, thus yielding (given \eqref{cnv.uek.2}) $\Xi_i = (-\Delta)^{(\beta+1)/4}\uk_i$. Thus, we have proved
\begin{align}\label{cnv.Leuek}
\vek_i \equiv \Le^{(\beta+1)/4}\uek_i\rightharpoonup (-\Delta)^{(\beta+1)/4}\uek_i \equiv \vk_i\quad\mbox{weakly in }L^2(Q_T).
\end{align}
Let us now consider
\begin{align*}
\int_{Q_T} |\nabla (-\Delta)^{-1} \Le^{(\beta+1)/4}\vek_i |^2 \dd x \dd t 
&= \int_0^T\|\nabla (-\Delta)^{-1} \Le^{(\beta+1)/4}\vek_i \|_{L^2(\Omega)}^2 \dd t.
\end{align*}
%
From \eqref{Le.spec} it follows
\begin{align*}
(-\Delta)^{-1} \Le^{(\beta+1)/4}\vek_i &= 
\sum_{j=1}^\infty \lambda_j^{(\beta+1)/4-1} g_{(\beta+1)/4}(\eps\lambda_j) (\vek_i,\phi_j)_{L^2(\Omega)}(-\Delta)^{-1}\lambda_j\phi_j \\
&=\sum_{j=1}^\infty \lambda_j^{(\beta+1)/4-1} g_{(\beta+1)/4}(\eps\lambda_j) (\vek_i,\phi_j)_{L^2(\Omega)}\phi_j ,
\end{align*}
and so
\begin{align*}
&\| \nabla (-\Delta)^{-1} \Le^{(\beta+1)/4}\vek_i \|_{L^2(\Omega)}^2\\
&= \sum_{j,k=1}^\infty \lambda_j^{(\beta+1)/4-1}\lambda_k^{(\beta+1)/4-1} g_{(\beta+1)/4}(\eps\lambda_j)g_{(\beta+1)/4}(\eps\lambda_k) (\vek_i,\phi_j)_{L^2(\Omega)}(\vek_i,\phi_k)_{L^2(\Omega)}
(\nabla \phi_j,\nabla\phi_k)_{L^2(\Omega)}
\end{align*}
but since $(\nabla \phi_j,\nabla\phi_k)_{L^2(\Omega)} = \lambda_j\delta_{jk}$, we deduce
\begin{align}\label{id.norm.eps}
\int_{Q_T}& |\nabla (-\Delta)^{-1} \Le^{(\beta+1)/4}\vek_i |^2 \dd x \dd t \\ \nonumber
&=\int_0^T\sum_{j=1}^\infty \lambda_j^{(\beta-1)/2} g_\alpha(\eps\lambda_j)^2  |(\vek_i,\phi_j)_{L^2(\Omega)}|^2 \dd t
\end{align}
However, $0\leq g_\alpha\leq 1$, $0<\beta<1$ and $\lambda_j\geq \lambda_1>0$ for all $j\geq 1$, so
\begin{align*}
\int_{Q_T}& |\nabla (-\Delta)^{-1} \Le^{(\beta+1)/4}\vek_i |^2 \dd x \dd t 
\leq \lambda_1^{(\beta-1)/2}\int_0^T\sum_{j=1}^\infty |(\vek_i,\phi_j)_{L^2(\Omega)}|^2 \dd t
= \lambda_1^{(\beta-1)/2}\|\vek_i\|_{L^2(Q_T)}^2
\end{align*}
and \eqref{ei.eps.1} as well as \eqref{cnv.Leuek} yield
\begin{align}\label{est.eps.1}
\int_{Q_T}& |\nabla (-\Delta)^{-1} \Le^{(\beta+1)/4}\vek_i |^2 \dd x \dd t \leq C.
\end{align}
From \eqref{Le.err}, \eqref{cnv.Leuek}, and \eqref{est.eps.1}, we deduce
\begin{align}\label{cnv.Leuek.2}
\nabla (-\Delta)^{-1} \Le^{(\beta+1)/4}\Le^{(\beta+1)/4}\uek_i\rightharpoonup
\nabla (-\Delta)^{(\beta-1)/2}\uk_i\quad\mbox{weakly in }L^2(Q_T).
\end{align}
Thanks to \eqref{cnv.uek} and \eqref{cnv.Leuek.2}, we can pass to the limit as $\eps\to 0$ in \eqref{3.app}--\eqref{3.app.ic} and obtain
\begin{align}\label{4.app}
	&\pa_t \uk_i - \kappa\Delta \uk_i
	+\kappa g[\uk_i] = \Div\left(
	\sum_{j=1}^n a_{ij} \uk_i\nabla (-\Delta)^{(\beta-1)/2}\uk_j\right), && (t,x) \in (0,+\infty) \times \Omega,\\
	\label{4.app.bc}
	&\kappa\nu\cdot\nabla \uk_i + \sum_{j=1}^n a_{ij}\uk_i\nu\cdot\nabla (-\Delta)^{(\beta-1)/2}\uk_j = 0, && (t,x) \in (0,+\infty)\times \partial\Omega,\\
	\label{4.app.ic}
	&\uk_i(0,\cdot) = u_{0,i}, && x \in \Omega.
\end{align}
We stress the fact that \eqref{4.app}--\eqref{4.app.bc} must be intended in the weak sense:
\begin{align}\label{4.app.w}
\begin{split}
&\int_0^T\langle \partial_t \uk_i , \phi \rangle \dd t + \kappa \int_0^T\int_\Omega\nabla \uk_i\cdot\nabla\phi \dd x \dd t + \kappa\int_0^T\int_\Omega g[\uk_i]\phi \dd x \dd t \\ 
&+ 
\int_0^T\int_\Omega\sum_{j=1}^n a_{ij} \uk_i\nabla (-\Delta)^{(\beta-1)/2}\uk_j\cdot\nabla\phi \dd x \dd t = 0, \qquad\text{ for all }\phi\in L^2((0,T); H^1(\Omega)).
\end{split} 
\end{align}

Also, via \eqref{cnv.Leuek} and Fatou's Lemma, we can take the (inferior) limit in \eqref{ei.eps.1} and deduce the following estimate:
\begin{align}\label{ei.kappa.1}
	H[& \uk(t)] + 4\kappa\sum_{i=1}^n \int_0^t\int_\Omega
	\pi_i |\nabla \sqrt{\uk_i}|^2 \dd x \dd \tau\\ \nonumber
	&+c_0\int_0^t\int_\Omega |\nabla (-\Delta)^{(\beta-1)/4}\uk|^2 \dd x \dd\tau\\ 
	\nonumber
	&+\kappa\sum_{i=1}^n\pi_i\int_0^t 
	\int_\Omega (\uk_i)^2\left(
	( \log (\uk_i) )_+ - \fint_\Omega (\log(\uk_i))_+ \dd y
	\right)_+ \dd x \dd\tau\\ \nonumber
	\leq & C(1+\kappa t) ,
\end{align}
where we also used the identity
$$
\int_0^t\int_\Omega |\nabla (-\Delta)^{(\beta-1)/4}\uk|^2 \dd x \dd\tau = 
\int_0^t\int_\Omega |(-\Delta)^{(\beta+1)/4}\uk|^2 \dd x \dd\tau
$$
which is easy to verify via the spectral decomposition of the fractional Laplacian (the argument is basically the same as the one employed to derive \eqref{id.norm.eps}).

\textbf{Step 4d.} \emph{Limit $\kappa \to 0^+$.}
The uniform bound for $\nabla (-\Delta)^{(\beta-1)/4}\uk$ in $L^2(Q_T)$ coming from \eqref{ei.kappa.1}, mass conservation and Lemmata \ref{lem.spfrSob}, \ref{lem.Poi} allow us to deduce via Aubin--Lions' lemma that 
\begin{align}\label{cnv.uk}
\uk_i\to u_i\quad\mbox{strongly in }L^2(Q_T),\\
\label{cnv.uk.2}
\nabla (-\Delta)^{(\beta-1)/4}\uk_i\rightharpoonup
\nabla (-\Delta)^{(\beta-1)/4}u_i\quad\mbox{weakly in }L^2(Q_T).
\end{align}
The uniform $L^2(Q_T)$ bounds for $\sqrt{\kappa}(\uk_i)\sqrt{(\log\uk_i)_+}$,
$\sqrt{\kappa}\nabla\sqrt{\uk_i}$ coming from \eqref{ei.kappa.1} and mass conservation allow us to infer that $\kappa g[\uk_i]\to 0$, $\kappa\nabla\uk_i\to 0$ strongly in $L^1(Q_T)$. We can therefore take the limit $\kappa\to 0$ in \eqref{4.app}--\eqref{4.app.ic} and obtain a weak solution to \eqref{1}--\eqref{1.ic}.

\textbf{Step 5.} \emph{Entropy estimate.}  Finally, we note that, taking the limit $\kappa \to 0$ in the $\kappa$-approximated entropy inequality \eqref{ei.kappa.1} yields 
\begin{align*}
H[u(t)] &+c_0\int_0^t\int_\Omega \sum_{i=1}^{n} |\nabla (-\Delta)^{(\beta-1)/4}u_i(t,x)|^2 \dd x \dd\tau\le C.
\end{align*}	

This concludes the proof of Theorem \ref{thm:ex}.
\end{proof}

\section{Weak-strong uniqueness}
\label{sec:thm-wsuniq}

In order to prove the weak-strong uniqueness principle, we shall rely on the entropy structure of the system (see \eqref{en}) and use the relative entropy method. 

\begin{proof}[Proof of Theorem~\ref{thm:wsuniq}]

We shall study the time-evolution of the following ``relative entropy'' between $u$ and $v$:
\begin{align*}
H[u\vert v] = H[u] - H[v]  - \langle H'[v], u-v\rangle = \sum_{i=1}^n\pi_i\int_\Omega\left(
u_i\log\frac{u_i}{v_i} - u_i + v_i
\right)\dd x.
\end{align*}
Rigorously, we need to argue at an approximate level and then pass to the limit in the approximated entropy inequalities. However, for brevity (since the approximation argument is very similar to the one presented in Theorem \ref{thm:ex}), we do here a formal computation. That is, we differentiate the relative entropy in time to compute: 
\begin{align*}
\frac{\d}{\d t} H[u\vert v] 
&= \sum_{i=1}^n\pi_i\int_\Omega\left(
\left(\log\frac{u_i}{v_i} \right)\pa_t u_i + \left(
1 - \frac{u_i}{v_i}\right)\pa_t v_i
\right)\dd x\\
&= -\sum_{i,j=1}^n\pi_i a_{ij}\int_\Omega\left(
u_i\nabla\log\frac{u_i}{v_i}\cdot\nabla (-\Delta)^{-\frac{1-\beta}{2}} u_j
- v_i\nabla\frac{u_i}{v_i}\cdot\nabla (-\Delta)^{-\frac{1-\beta}{2}}v_j
\right)\dd x\\
&= -\sum_{i,j=1}^n\pi_i a_{ij}\int_\Omega\left(
\nabla u_i\cdot\nabla (-\Delta)^{-\frac{1-\beta}{2}} u_j
-\frac{u_i}{v_i}\nabla v_i\cdot\nabla (-\Delta)^{-\frac{1-\beta}{2}} u_j
\right.\\
&\left.\qquad\qquad 
-\nabla u_i\cdot\nabla (-\Delta)^{-\frac{1-\beta}{2}}v_j
+\frac{u_i}{v_i}\nabla v_i\cdot\nabla (-\Delta)^{-\frac{1-\beta}{2}}v_j
\right)\dd x\\
&= -\sum_{i,j=1}^n\pi_i a_{ij}\int_\Omega\left(
\nabla u_i\cdot\nabla (-\Delta)^{-\frac{1-\beta}{2}} u_j
-\nabla v_i\cdot\nabla (-\Delta)^{-\frac{1-\beta}{2}} u_j\right.\\
&\left.\qquad\qquad -\nabla u_i\cdot\nabla (-\Delta)^{-\frac{1-\beta}{2}} v_j
+\nabla v_i\cdot\nabla (-\Delta)^{-\frac{1-\beta}{2}} v_j
\right)\dd x\\
&\quad -\sum_{i,j=1}^n\pi_i a_{ij}\int_\Omega
\frac{u_i-v_i}{v_i}\nabla v_i\cdot\nabla (-\Delta)^{-\frac{1-\beta}{2}}(v_j-u_j)\dd x.
\end{align*}
At this point, we use the fact that $\pi_i a_{ij} = \pi_j a_{ji}$ to infer
\begin{align*}
\frac{\d}{\d t} H[u\vert v] &= -\sum_{i,j=1}^n\pi_i a_{ij}\int_\Omega
(-\Delta)^{-\frac{1-\beta}{4}}\nabla(u_i-v_i)\cdot
(-\Delta)^{-\frac{1-\beta}{4}}\nabla(u_j-v_j) \dd x\\
&\quad +\sum_{i,j=1}^n\pi_i a_{ij}\int_\Omega
\frac{u_i-v_i}{v_i}\nabla v_i\cdot\nabla (-\Delta)^{-\frac{1-\beta}{2}}(u_j-v_j)\dd x.
\end{align*}
Since $\pi_i a_{ij}$ is positive definite and its smallest eigenvalue is $\alpha>0$, we deduce
\begin{align*}
\frac{\d}{\d t}& H[u\vert v] + \alpha\sum_{i=1}^n\int_\Omega
|(-\Delta)^{-\frac{1-\beta}{4}}\nabla(u_i-v_i)|^2 \dd x\\
&\leq \sum_{i,j=1}^n\pi_i a_{ij}\int_\Omega
\frac{u_i-v_i}{v_i}\nabla v_i\cdot\nabla (-\Delta)^{-\frac{1-\beta}{2}}(u_j-v_j)\dd x\\
&= \sum_{i,j=1}^n\pi_i a_{ij}\int_\Omega
(-\Delta)^{-\frac{1-\beta}{4}}
\left( \frac{u_i-v_i}{v_i}\nabla v_i \right)\cdot\nabla (-\Delta)^{-\frac{1-\beta}{4}}(u_j-v_j)\dd x.
\end{align*}
Young's inequality and the boundedness of $(-\Delta)^{-\frac{1-\beta}{4}} : L^2\to L^2$ yield
\begin{align*}
\frac{\d}{\d t}H[u\vert v] &+ \frac{\alpha}{2}\sum_{i=1}^n\int_\Omega
|(-\Delta)^{-\frac{1-\beta}{4}}\nabla(u_i-v_i)|^2 \dd x
\leq C \int_\Omega \left|(-\Delta)^{-\frac{1-\beta}{4}}
\left( \frac{u_i-v_i}{v_i}\nabla v_i \right)\right|^2 \dd x.
\end{align*}
Employing Corollary \ref{Hminusr.coro} leads to
\begin{align*}
\frac{\d}{\d t}H[u\vert v] &+ \frac{\alpha}{2}\sum_{i=1}^n\int_\Omega
|(-\Delta)^{-\frac{1-\beta}{4}}\nabla(u_i-v_i)|^2 \dd x
\leq C \| (u_i-v_i) \nabla\log v_i\|_{L^p(\Omega)}^2 ,\quad 
\frac{2d}{d+1-\beta}<p.
\end{align*}
Lemma \ref{Sob.emb.spfr} leads to the following Gagliardo--Nirenberg--Sobolev's estimate
\begin{align*}
\|u_i-v_i\|_{L^{q}}\leq C
\|u_i-v_i\|_{L^1(\Omega)}^{1-\xi}
\|(-\Delta)^{-\frac{1-\beta}{4}}\nabla(u_i-v_i)\|_{L^2(\Omega)}^\xi,\quad  p<q<\frac{2d}{d-1-\beta}, \quad 
\xi = \frac{2d}{d+1+\beta}\left(1-\frac{1}{q}\right).
\end{align*}
From the above inequality as well as H\"older and Young's inequalities we get
\begin{align*}
&\| (u_i-v_i) \nabla\log v_i\|_{L^p(\Omega)}^2\\
&\leq 
\left(\int_\Omega |u_i-v_i|^{q} \dd x\right)^{2/q}
\left( \int_\Omega |\nabla\log v_i|^{pq/(q-p)} \dd x \right)^{
\frac{2(q-p)}{qp} }\\
&\leq C\left( \int_\Omega |\nabla\log v_i|^{pq/(q-p)} \dd x \right)^{\frac{2(q-p)}{qp}}
\|u_i-v_i\|_{L^1(\Omega)}^{2(1-\xi)}
\|(-\Delta)^{-\frac{1-\beta}{4}}\nabla(u_i-v_i)\|_{L^2(\Omega)}^{2\xi}\\
&\leq C\left( \int_\Omega |\nabla\log v_i|^{pq/(q-p)} \dd x \right)^{
\frac{2(q-p)}{qp(1-\xi)}	}
\|u_i-v_i\|_{L^1(\Omega)}^{2} + 
\frac{\alpha}{4}\int_\Omega
|(-\Delta)^{-\frac{1-\beta}{4}}\nabla(u_i-v_i)|^2 \dd x
\end{align*}
which implies
\begin{align*}
\frac{\d}{\d t}H[u\vert v] &+ \frac{\alpha}{4}\sum_{i=1}^n\int_\Omega
|(-\Delta)^{-\frac{1-\beta}{4}}\nabla(u_i-v_i)|^2 \dd x
\leq C\sum_{i=1}^n\|\nabla\log v_i\|_{L^{q_2}(\Omega)}^{q_1}
\|u_i-v_i\|_{L^1(\Omega)}^{2},
\end{align*}
with 
$$
q_1 = \frac{2}{1-\xi} = (d+1+\beta)\left(\frac{1+\beta-d}{2}+\frac{d}{q} \right)^{-1},\quad
q_2 = \frac{pq}{q-p}.
$$
Employing the lower bound $\frac{2d}{d+1-\beta}<p$ and performing a straightforward computation leads to 
$$\frac{d+1+\beta}{q_1} + \frac{d}{q_2} < 1,$$ 
while Csisz\'ar--Kullback--Pinsker's inequality (see, e.\,g., \cite[Theorem 1.1]{MR1801751}) allows us to deduce
\begin{align*}
\frac{\d}{\d t}H[u\vert v] &+ \frac{\alpha}{4}\sum_{i=1}^n\int_\Omega
|(-\Delta)^{-\frac{1-\beta}{4}}\nabla(u_i-v_i)|^2 \dd x
\leq C\sum_{i=1}^n\|\nabla\log v_i\|_{L^{q_2}}^{q_1}
H[u\vert v].
\end{align*}
Assumption \eqref{hp.wsuni} and Gr\"onwall's lemma imply that $H[u(t)\vert v(t)]=0$ for $t>0$, meaning that $u(t)=v(t)$ a.\,e.~in $\Omega$, for $t>0$. This concludes the proof.
\end{proof}

\section{Long-time asymptotics}
\label{sec:thm-lt}

Using the relative entropy approach of Section \ref{sec:thm-wsuniq}, we are ready to prove the convergence of $u(t,\cdot)$ to the equilibrium $u^\infty$ as $t \to +\infty$.

\begin{proof}[Proof of Theorem \ref{thm:lt}]
From \eqref{en.in} (see Theorem \ref{thm:ex}), it follows
\begin{align*}
\frac{\d}{\d t}H[u\vert u^\infty] \leq -c_0\int_\Omega\sum_{i=1}^n |\na (-\Delta)^{-\frac{1-\beta}{4}}u_i|^2 \dd x,
\end{align*}
where $H[u\vert u^\infty]$ is the relative entropy.
From Lemma \ref{lem.Poi} (with $r=(1+\beta)/2\in (0,1)$), we deduce 
\begin{align*}
\frac{\d}{\d t}H[u\vert u^\infty] \leq -C\int_\Omega\sum_{i=1}^n |u_i - u_i^\infty |^2 \dd x,
\end{align*}
which yields
\begin{align*}
\frac{\d}{\d t}H[u\vert u^\infty] \leq -C H[u\vert u^\infty],\quad t>0.
\end{align*}
Then, by Gr\"onwall's inequality, we deduce $H[u(t)\vert u^\infty]\leq e^{-Ct}H[u_0\vert u^\infty]$ for $t>0$, which  implies the strong convergence of $u(t)$ towards $u^\infty$ in $L^1(\Omega)$ with exponential rate via a  Csisz\'ar--Kullback--Pinsker's inequality (see, e.\,g., \cite[Theorem 1.1]{MR1801751}). This concludes the proof of Theorem~\ref{thm:lt}.
\end{proof}

\vspace{2mm}

\appendix

\section*{Acknowledgments}

N.~De Nitti is a member of the Gruppo Nazionale per l’Analisi Matematica, la Probabilità e le loro Applicazioni (GNAMPA) of the Istituto Nazionale di Alta Matematica (INdAM). 

Nicola Zamponi thanks the University of Augsburg for financial support via the research program “Forschungspotentiale besser nutzen!”.

We thank E.~Zuazua for suggesting the problem and for many insightful discussions. We also acknowledge the anonymous referees for their constructive comments.

\vspace{2mm}

\bibliographystyle{abbrv} 
\bibliography{CDSpectral-ref.bib}
\vfill

\end{document}